\documentclass[leqno,a4paper,11pt]{amsart}
\usepackage{amsmath,amsthm,amssymb}
\usepackage[utf8]{inputenc}
\usepackage[T1]{fontenc}

\usepackage[margin=3.2cm]{geometry}
\usepackage{enumerate}
\usepackage{bbm}
\usepackage{todonotes}
\usepackage{mathrsfs}
\usepackage{mathabx}
\usepackage{hyperref}
\usepackage{nicefrac}
\usepackage{tikz}
\usetikzlibrary{matrix}
\usepackage{comment}
\usepackage{mathtools}
\usepackage{url}
\makeatletter
\g@addto@macro{\UrlBreaks}{\UrlOrds}
\makeatother

\numberwithin{equation}{section}

\usepackage[sort,numbers]{natbib}

\providecommand{\noopsort}[1]{}

\def\qed{\unskip\quad \hbox{\vrule\vbox
to 6pt {\hrule width 4pt\vfill\hrule}\vrule} }

\newcommand{\bez}{\nopagebreak\hspace*{\fill}
 \nolinebreak$\qed$\vspace{5mm}\par}

\newtheorem{Th}{Theorem}[section]
\newtheorem{Prop}[Th]{Proposition}
\newtheorem{Lemma}[Th]{Lemma}
\newtheorem{Conj}{Conjecture}
\theoremstyle{definition}
\newtheorem{Remark}[Th]{Remark}

\newtheorem{Cor}[Th]{Corollary}

\newcommand{\beq}{\begin{equation}}
\newcommand{\eeq}{\end{equation}}

\def\scalar(#1,#2){(#1\mid#2)}

\renewcommand{\hat}{\widehat}

\newcommand{\raz}{\mathbbm{1}}
\newcommand{\Leb}{{\mathrm{Leb}}}

\newcommand{\ot}{\otimes}
\newcommand{\ov}{\overline}

\newcommand{\bs}{\mathbb{S}}
\newcommand{\A}{\mathbb{A}}

\newcommand{\R}{{\mathbb{R}}}
\newcommand{\T}{{\mathbb{T}}}
\newcommand{\C}{{\mathbb{C}}}
\newcommand{\Z}{{\mathbb{Z}}}
\newcommand{\N}{{\mathbb{N}}}
\newcommand{\E}{{\mathbb{E}}}
\newcommand{\EE}{{\mathbb{E}}}
\newcommand{\PP}{{\mathbb{P}}}

\newcommand{\vep}{\varepsilon}

\newcommand{\M}{{\cal M}}

\newcommand{\lio}{\boldsymbol{\lambda}}

\let\oldpmod\pmod
\renewcommand{\pmod}[1]{\hspace{-0.15cm}\oldpmod {#1}}

\usepackage[normalem]{ulem}

\usepackage[makeroom]{cancel}

\begin{document}
\vspace*{-3em}
\title{On the local Fourier uniformity problem for small sets}
\author[Kanigowski]{Adam Kanigowski}
\address{Department of Mathematics, University of Maryland, College Park, MD USA}
\email{akanigow@umd.edu}

\author[Lema\'nczyk]{Mariusz Lema\'nczyk}
\address{Faculty of Mathematics and Computer Science, Nicolaus Copernicus University, Chopin street 12/18,
87-100 Toru\'n, Poland}
\email{mlem@mat.umk.pl}

\author[Richter]{Florian K. Richter}
\address{\'Ecole Polytechnique F\'ed\'erale de Lausanne (EPFL), Lausanne, Vaud, Switzerland}
\email{f.richter@epfl.ch}

\author[Ter\"av\"ainen]{Joni Ter\"av\"ainen}

\address{Department of Mathematics and Statistics, University of Turku, 20014 Turku, Finland}
\email{joni.p.teravainen@gmail.com}

\begin{abstract}
We consider vanishing properties of exponential sums of the Liouville function $\lio$ of the form
$$
\lim_{H\to\infty}\limsup_{X\to\infty}\frac{1}{\log X}\sum_{m\leq X}\frac{1}{m}\sup_{\alpha\in C}\bigg|\frac{1}{H}\sum_{h\leq H}\lio(m+h)e^{2\pi ih\alpha}\bigg|=0,
$$
where $C\subset\T$.
The case $C=\T$ corresponds to the local $1$-Fourier uniformity conjecture of Tao, a central open problem in the study of multiplicative functions
with far-reaching number-theoretic applications.
We show that the above holds for any closed set $C\subset\T$ of zero Lebesgue measure. Moreover, we prove that extending this to any set $C$ with non-empty interior is equivalent to the $C=\T$ case, which shows that our results are essentially optimal without resolving the full conjecture.

We also consider higher-order variants. We prove that if the linear phase $e^{2\pi ih\alpha}$ is replaced by a polynomial phase $e^{2\pi ih^t\alpha}$ for $t\geq 2$ then the statement remains true for any set $C$ of upper box-counting dimension $< 1/t$.
The statement also remains true if the supremum over linear phases is replaced with a supremum over all nilsequences coming form a compact countable ergodic subsets of any $t$-step nilpotent Lie group.
Furthermore, we discuss the unweighted version of the local $1$-Fourier uniformity problem, showing its validity for a class of ``rigid''  sets (of full Hausdorff dimension) and proving a density result for all closed subsets of zero Lebesgue measure.
\end{abstract}

\maketitle

\section{Introduction} The aim of this paper is to establish new results concerning the local $t$-Fourier uniformity conjecture over sets of measure zero resulting from recent progress in our understanding of the  Chowla and Sarnak conjectures. Throughout, let $\lio(n)=(-1)^{\Omega(n)}$ denote the Liouville function, where $\Omega(n)$ is the number of prime divisors of $n$ (counted with multiplicities).

\subsection{\texorpdfstring{Local $t$-Fourier uniformity}{Local t-Fourier uniformity}}
A $t$-step nilmanifold is a quotient space $G/\Gamma$, where $G$ is a $t$-step  nilpotent Lie group and $\Gamma$ is a discrete cocompact subgroup of $G$.
For technical reasons, we assume throughout this work that every nilpotent Lie group under consideration is either connected or spanned by the connected component of the identity element and finitely many other group elements.\footnote{This, or similar, restrictions on $G$ are a standard convention when studying nilsystems in ergodic theory and encompasses most relevant examples, see \cite[Subsection 2.1]{Le} or \cite[p.~155]{Ho-Kr} together with Appendix~\ref{a:nowy}.}
For $f\colon \N\to\C$, we write
\[
\mathbb{E}_{m\leq M}f(m)=\frac{1}{M}\sum_{m\leq M}f(m)
\quad\text{and}\quad
\mathbb{E}^{\rm log}_{m\leq M}f(m)=\frac{1}{\log M}\sum_{m\leq M}\frac{f(m)}{m}
\]
for the Ces\`aro and logarithmic averages of $f$, respectively.
The local $t$-Fourier uniformity conjecture of Tao states the following.

\begin{Conj}[{\cite[Conjecture 1.7]{Ta1}}]\label{conj:fourier} Let $t\in \mathbb{N}$. For each $t$-step nilmanifold $G/\Gamma$ and any $f\in C(G/\Gamma)$, we have\footnote{We remark that Tao's original formulation of the conjecture assumes $G$ to be connected and simply connected and $f$ to be Lipschitz continuous.
The restriction on $f$ can be relaxed since the space of Lipschitz functions on $G/\Gamma$ is dense in the space of continuous functions by the Stone--Weierstrass theorem.
Concerning the restriction on $G$, it follows from Proposition~\ref{p:embeddingnilLie} that Conjecture~\ref{conj:fourier} for connected and simply connected $G$ is equivalent to our formulation.
}
\beq\label{unif1}
\lim_{H\to\infty}\limsup_{M\to \infty}\EE_{m\leq M}\sup_{g\in G}\Big|\EE_{h\leq H}\lio(m+h)f(g^h\Gamma)\Big|=0\eeq
or, for  the logarithmic averages,
\beq\label{uniflog1}
\lim_{H\to\infty}\limsup_{M\to \infty}\EE^{\rm log}_{m\leq M}\sup_{g\in G}\Big|\EE_{h\leq H}\lio(m+h)f(g^h\Gamma)\Big|=0.\eeq
\end{Conj}

As shown by Tao~\cite[Theorem 1.8]{Ta1}, the validity of the logarithmic local $t$-Fourier uniformity conjecture~\eqref{uniflog1} for all $t\geq1$ is equivalent to two important conjectures in multiplicative number theory, namely the logarithmic Chowla conjecture  on autocorrelations of the Liouville function and the logarithmically averaged version of Sarnak's M\"obius orthogonality conjecture. We recall that the logarithmically averaged Chowla conjecture is the statement that for any $k\in \mathbb{N}$ and any natural numbers $h_1<\ldots<h_k$, we have
\begin{align}\label{eq:logChowla}
 \lim_{M\to \infty}\mathbb{E}_{m\leq M}^{\log}\lio(m+h_1)\cdots \lio(m+h_k)=0.
\end{align}
The logarithmically averaged Sarnak conjecture in turn is the statement that for any deterministic sequence $a\colon\mathbb{N}\to \mathbb{C}$, we have
\begin{align*}
\lim_{M\to \infty}\mathbb{E}_{m\leq M}^{\log}\lio(m)a(m)=0.
\end{align*}
See for example the survey~\cite{Fe-Ku-Le} for a discussion of these conjectures and for some of the progress made towards them.

\subsection{Local 1-Fourier uniformity for small sets}
The local $t$-Fourier uniformity problem is still open, and even the case $t=1$  seems to be out of reach using present techniques. By Fourier expansion, the local $1$-Fourier uniformity problem is equivalent to
\beq\label{unif5}\lim_{H\to\infty}\limsup_{M\to \infty}\EE_{m\leq M}\sup_{\alpha\in \T}\big|\EE_{h\leq H}\lio(m+h)e(h\alpha)\big|=0,\eeq
where we use the standard notation $e(t)=e^{2\pi it}$ for $t\in\R$.
This was proved in the regime $H\geq M^{\varepsilon}$ for any fixed $\varepsilon>0$ in~\cite{Ma-Ra-Ta}, and improved to $H\geq \exp((\log M)^{\theta})$ for any fixed $\theta>5/8$ in~\cite{Ma-Ra-Ta-Te-Zi}, and very recently further to $H\geq \exp(C(\log M)^{1/2}(\log \log M)^{1/2})$ for some $C>0$ in~\cite{Walsh}.

Until a few years ago,~\eqref{unif5} was known to hold only in the case when the supremum in $\alpha$ is taken over a finite set, which follows from the work of Matom\"aki--Radziwi\l{}\l{}--Tao~\cite{MRT}.  McNamara~\cite{McN} was the first to improve on this result in the logarithmic case by showing that for all (closed\footnote{By continuity, the problem of taking $\sup_{C}$ is the same as taking $\sup_{\overline{C}}$, so all the results in the paper are about closed subsets.}) sets $C\subset\T$ of box-counting dimension $<1$, we have
\beq\label{unif5b}\lim_{H\to\infty}\limsup_{M\to \infty}\EE^{\rm log}_{m\leq M}\sup_{\alpha\in C}\big|\EE_{h\leq H}\lio(m+h)e(h\alpha)\big|=0.\eeq
McNamara also gave an example of $C$ satisfying~\eqref{unif5b} and of full Hausdorff dimension.  A larger class of $C$ satisfying~\eqref{unif5b} was provided by Huang--Xu--Ye~\cite{Hu-Xu-Ye} by considering the class of closed subsets whose packing dimension is $<1$. Additionally, they  provided
the first example of an infinite, closed, and uncountable subset $C$ of $\T$ for which
the non-logarithmic version of~\eqref{unif5b} holds. More precisely, they showed
\beq\label{unif6}\lim_{H\to\infty}\limsup_{M\to \infty}\EE_{m\leq M}\sup_{\alpha\in C}\big|\EE_{h\leq H}\lio(m+h)e(h\alpha)\big|=0,\eeq
for all sets $C$ of packing dimension~0.
We note that all sets considered in~\cite{Hu-Xu-Ye} and~\cite{McN} are closed with zero Lebesgue measure.

Our first result is the following.
\begin{Cor}\label{p:uniflog1} For each closed $C\subset\T$ with ${\rm Leb}(C)=0$, the logarithmic local 1-Fourier uniformity~\eqref{unif5b} holds.\end{Cor}

We also show (in Section~\ref{sec:implication}) that the restriction to sets of measure zero in Corollary~\ref{p:uniflog1} is crucial, as relaxing this condition somewhat would lead to a 
resolution of the full logarithmic local $1$-Fourier uniformity conjecture. 

\begin{Th}\label{th:implication} Suppose that there exists a set $C\subset \mathbb{T}$ with non-empty interior such that the logarithmic local $1$-Fourier uniformity~\eqref{unif5b} holds for $C$. Then the same holds for $C=\mathbb{T}$.
\end{Th}

We will also show in Theorem~\ref{th:assum} below that a slight extension of Corollary~\ref{p:uniflog1} (allowing Dirichlet character twists, which we can handle with the same argument) cannot be extended to \emph{any} positive measure set without settling the logarithmic local $1$-Fourier uniformity conjecture in full. 

Corollary~\ref{p:uniflog1} is a special case of Theorem~\ref{p:jt1} below which deals with Ces\`aro averages instead of logarithmic averages. To state this theorem, we introduce the following notation.
For a set $\mathcal{M}=\{M_1,M_2,M_3,\ldots\}\subset \mathbb{N}$ with $M_1<M_2<M_3<\ldots$ and a function $f\colon \mathcal{M}\to \mathbb{C}$,  write
\begin{align*}
\limsup_{\substack{M\in \mathcal{M}\\M\to \infty}}f(M)=\limsup_{i\to \infty}f(M_i)
\quad\text{and}\quad
  \lim_{\substack{M\in \mathcal{M}\\M\to \infty}}f(M)=\lim_{i\to \infty}f(M_i),
\end{align*}
where the latter is only defined when the limit on the right-hand side exists.

\begin{Th}\label{p:jt1} There exists a set $\mathcal{M}\subset \mathbb{N}$ of logarithmic density\footnote{The logarithmic density $\delta(\mathcal{M})$ of a set $\mathcal{M}\subset \mathbb{N}$ is $\lim_{M\to \infty}\mathbb{E}_{m\leq M}^{\log}\raz_{\mathcal{M}}(m)$ (when this limit exists).} $1$ such that the following holds.
Let $C\subset\T$ be any closed set with $\Leb(C)=0$. Then we have
\begin{align}\label{eq:cesaro}
\lim_{H\to \infty}\limsup_{\substack{M\in \mathcal{M}\\M\to \infty}}\mathbb{E}_{m\leq M}\sup_{\alpha\in C}\Big|\EE_{m\leq h< m+H}\lio(h)e(h\alpha)\Big|=0.
\end{align}
\end{Th}

Corollary~\ref{p:uniflog1} follows from Theorem~\ref{p:jt1} by partial summation\footnote{Indeed, by partial summation, for any bounded sequence $a\colon \N\to \R$, we have $$\limsup_{M\to \infty}\frac{1}{\log M}\sum_{m\leq M}\frac{a(m)}{m}=\limsup_{M\to \infty}\frac{1}{\log M}\sum_{k\leq M}k^{-2}\sum_{m\leq k}a(m).$$ The claim follows by applying this with $a(m)$ being the sequence inside the averaging operator in~\eqref{eq:cesaro}}. To obtain the  Ces\`aro statement~\eqref{unif6}, we need to put some further restrictions on $C$.

\begin{Th}\label{p:unif1} Assume that $C\subset\T$ is a closed set for which there is a sequence $(q_n)$ of natural numbers such that\footnote{Given $t\in \mathbb{R}$, $\|t\|$ stands for the distance of $t$ to the nearest integer(s).}
\beq\label{sztywnosc}\lim_{n\to \infty}\|q_n\alpha\|=0\text{ for  each }\alpha\in C\eeq
and
\beq\label{bpv}(q_n)\text{ has bounded prime volume, i.e. }\sup_n\sum_{\substack{p\in \PP\\ p\mid q_n}}\frac1p<+\infty.\eeq
Then~\eqref{unif6} holds. \end{Th}

We will show in Appendix~\ref{App:B} that there exist sets $C\subset \T$ of full Hausdorff dimension satisfying~\eqref{sztywnosc} and~\eqref{bpv}.

\subsection{\texorpdfstring{Local polynomial $t$-Fourier uniformity for small sets}{Local polynomial t-Fourier uniformity for small sets}}

We now consider the $t\geq 2$ case of Conjecture~\ref{conj:fourier}, with the supremum over $g$ being taken over a sparse set. An important special case is that when $G/\Gamma$ is isomorphic to a torus $\T^d$, $d\in\N$; then, by Fourier expansion, the claim is equivalent to
\beq\label{unif5c}\lim_{H\to\infty}\limsup_{M\to \infty}\EE_{m\leq M}\sup_{\deg(P)\leq t}\big|\EE_{h\leq H}\lio(m+h)e(P(h))\big|=0,\eeq
where the supremum is over polynomials $P(X)\in \mathbb{R}[X]$ of degree at most $t$.
See~\cite{Ma-Ra-Ta-Te-Zi} for a result establishing this in the regime $H\geq \exp((\log M)^{\theta})$, for any fixed $\theta>5/8$. Less is known in the case $t\geq 2$ compared to the $t=1$ case about statements of the form
\beq\label{unif5d}\lim_{H\to\infty}\limsup_{M\to \infty}\EE^{\rm log}_{m\leq M}\sup_{\alpha \in C}\big|\EE_{h\leq H}\lio(m+h)e(\alpha h^t)\big|=0.\eeq
To our knowledge, the only previous result here is the case where $C$ is finite; this follows from~\cite[Theorem 5]{Ab-Le-Ru}. We can improve on this by showing that any closed set $C$ of box-counting dimension $<1/t$ has this property. Recall that the upper box-counting dimension of a set $C\subset \T$ is defined as the infimum over all $s\geq 0$ such that
\begin{align*}
 \limsup_{J\to \infty}\frac{\min\{k\geq 1\colon \,\, C\,\textnormal{ can be covered by }\, k\, \textnormal{ intervals of length}\leq 1/J\}}{J^{s}}=0.
\end{align*}
The lower box-counting dimension of $C$ is defined similarly with $\liminf$ in place of $\limsup$, but we will not make use of this notion.

\begin{Th}\label{thm:packingdimension} Let $t\geq 2$. There exists a set $\mathcal{M}\subset \mathbb{N}$ of logarithmic density $1$ such that for any closed $C\subset \T$ of upper box-counting dimension~$<1/t$,
we have
\beq
\label{eqn_packdim_1}\lim_{H\to\infty}\limsup_{\substack{M\in \mathcal{M}\\M\to \infty}}~\EE_{m\leq M}\sup_{\alpha\in C}\big|\EE_{h\leq H}\lio(m+h)e(\alpha h^t)\big|=0.\eeq
\end{Th}

Again by partial summation, we can also obtain a version of~\eqref{eqn_packdim_1}, where we use $\limsup_{M\to \infty}$ and a logarithmic average $\mathbb{E}_{m\leq M}^{\log}$.

\subsection{\texorpdfstring{A stronger local $t$-Fourier uniformity problem}{A stronger local t-Fourier uniformity problem}}

We now turn to the case of general nilpotent Lie groups $G$.
Together with~\eqref{unif1} and~\eqref{uniflog1}, we could consider their stronger and more symmetric versions:
\begin{align}\label{unif1b}\begin{split}
\lim_{H\to\infty}\limsup_{M\to \infty}\EE_{m\leq M}\sup_{g,g'\in G}\big|\EE_{h\leq H}\lio(m+h)f(g^{m+h}g'\Gamma)\big|=0\end{split}\end{align}
 and
\beq\label{uniflog1b}
\lim_{H\to\infty}\limsup_{M\to \infty}\EE^{\rm log}_{m\leq M}\sup_{g,g'\in G}\big|\EE_{h\leq H}\lio(m+h)f( g^{m+h}g'\Gamma)\big|=0\eeq
for all $f\in C(G/\Gamma)$. See~\cite{Ma-Ra-Ta-Te-Zi} for a result proving this\footnote{By \cite[Theorem~4.3]{Ma-Ra-Ta-Te-Zi} and the non-pretentiousness of the Liouville function (\cite[(1.12)]{MRT}), for any $\varepsilon>0$ one has
$\mathbb{E}_{m\leq M}\sup_g|\mathbb{E}_{h\leq H}\lio(m+h)f(g(m+h)\Gamma)|=o_{M\to \infty}(1)$ in the regime $M\geq H^{\varepsilon}$,
where the supremum is over all polynomial sequences $g\colon \mathbb{Z
}\to G$. Specializing to polynomial sequences of the form $n\mapsto g^ng'$ with $g,g'\in G$,  we get a similar supremum as in~\eqref{unif1b}.} in the regime $M\geq H^{\varepsilon}$ with $\varepsilon>0$ fixed.
These symmetric versions have rather neat dynamical reformulations (see the strong LOMO property below), and in the case $t=1$ they are equivalent to~\eqref{unif1} and~\eqref{uniflog1}, respectively,  as it is enough to consider $f$ being a character of $G/\Gamma$. Moreover, by Tao's work~\cite{Ta1}, the statement~\eqref{uniflog1} implies the logarithmic Chowla conjecture~\eqref{eq:logChowla}, and also~\eqref{eq:logChowla} implies~\eqref{uniflog1b}, so~\eqref{uniflog1} and~\eqref{uniflog1b} turn out to be equivalent.\footnote{The implication from~\eqref{uniflog1} to~\eqref{eq:logChowla} follows from~\cite[Theorem 1.8 and Remark 1.9]{Ta1}. For the implication from~\eqref{eq:logChowla} to~\eqref{uniflog1b}, note that the proof in~\cite{Ta1} that the logarithmic Chowla conjecture~\eqref{eq:logChowla} implies~\eqref{uniflog1} works equally well to show that~\eqref{eq:logChowla} implies~\eqref{uniflog1b} (or even a more general version in which $g^{m+h}g'$ is replaced with $g(m+h)$, where $g(\cdot)$ is any polynomial sequence from $\mathbb{Z}$ to $G$).}

Despite the equivalence of~\eqref{uniflog1} and~\eqref{uniflog1b}, for $t\geq 2$ partial progress on~\eqref{uniflog1b} with the supremum over a small set is harder to obtain than for~\eqref{uniflog1}. This is already seen in the case of abelian $G$, where we are now interested in sets $C\subset \mathbb{T}$ for which we can show
\beq\label{unif7a}
\lim_{H\to\infty}\limsup_{M\to \infty}\EE^{\rm log}_{m\leq M}\sup_{\substack{P(n)=\alpha n^{t}+Q(n)\\
\alpha\in C,\,{\rm deg}(Q)\leq t-1}}\big|\EE_{m\leq h<m+H}\lio(h)e(P(h))\big|=0.\eeq
For this problem, one can show (see Subsection~\ref{ss:zeronalezy}) that if~\eqref{unif7a} holds for some infinite, closed $C$ containing a rational number, then~\eqref{unif7a} holds with $t-1$ in place of $t$ for the full set $C=\T$. Hence, we cannot hope to be able to show~\eqref{unif7a} for very ``large'' infinite sets (in particular those that contain at least one rational number). On the other hand, whenever $C$ is countably infinite and contains no rational numbers, we are able to prove~\eqref{unif7a}.
In fact~\eqref{unif7a} for such $C$ is a straightforward consequence of the following theorem:

\begin{Th}\label{p:uniflog2A} Let $t\in \N$.  Let $G/\Gamma$ be a $t$-step nilmanifold and let $f\in C(G/\Gamma)$. For each countable compact subset $C\subset G$  for which for all $g\in C$ the nil-rotation $g'\Gamma\mapsto gg'\Gamma$ is ergodic,\footnote{We recall Leibman's condition (Thm.\ 2.17 in~\cite{Le}): If $G$ is generated by  its connected component $G^{\circ}$ and $g$, then the translation by $g$ on $G/\Gamma$ is ergodic if and only if it is ergodic on the torus $G/([G,G]\Gamma)$.}
\beq
\lim_{H\to\infty}\limsup_{M\to \infty}\EE^{\rm log}_{m\leq M}\sup_{g\in C,g'\in G}\Big|\EE_{h\leq H}\lio(m+h)f(g^{m+h}g'\Gamma)\Big|=0.
\eeq
In particular, the local $t$-Fourier uniformity holds on $C$.
\end{Th}

As a corollary, we get that~\eqref{unif7a} holds for sets $C$ as in the above theorem.

\begin{Cor}\label{p:uniflog2} Let $t\geq2$. For each countable, closed subset $C\subset\T$ of  irrational numbers, we have
\beq\label{unif7c}
\lim_{H\to\infty}\limsup_{M\to \infty}\EE^{\rm log}_{m\leq M}\sup_{\substack{P(n)=\alpha n^{t}+Q(n)\\
\alpha\in C,\,{\rm deg}(Q)\leq t-1}}\Big|\EE_{ h\leq H}\lio(m+h)e(P(m+h))\Big|=0.
\eeq
\end{Cor}

\subsection{A dynamical interpretation and the strong LOMO property}

In order to see the relationship of the stronger
local $t$-Fourier uniformity statements~\eqref{unif1b},~\eqref{uniflog1b}
with dynamics, more precisely with Sarnak's conjecture~\cite{Sa}, recall first the concept of strong LOMO\footnote{Acronym of ``Liouville orthogonality of moving orbits''.} (or logarithmic strong LOMO) introduced and studied in~\cite{Ab-Le-Ru},~\cite{Ab-Ku-Le-Ru},~\cite{Go-Le-Ru},~\cite{Ka-Ku-Le-Ru}. Given a homeomorphism $T$ of a compact metric space $X$, we say that it satisfies the strong LOMO property if for all increasing sequences $(b_k)\subset\N$ with density\footnote{The density of $A\subset\N$ is defined as $d(A)=\lim_{M\to\infty}\EE_{m\leq M}\raz_A(m)$ (when this exists).} $d(\{b_k\colon\, k\geq1\})=0$, all sequences $(x_k)\subset X$ and all $f\in C(X)$, we have
\beq\label{smomo1}\lim_{K\to\infty}\frac1{b_K}\sum_{k<K}\Big|\sum_{b_k\leq n< b_{k+1}}f(T^{n}x_k)\lio(n)\Big|=0\eeq
or in its logarithmic form (here we assume that $\delta(\{b_k\colon \,k\geq1\})=0$)
\beq\label{smomolog1}\lim_{K\to\infty}\frac1{\log b_K}\sum_{k<K}\Big|\sum_{b_k\leq n< b_{k+1}}\frac{1}{n}f(T^{n}x_k)\lio(n)\Big|=0.\eeq
Even though the strong LOMO property looks much stronger than the Liouville orthogonality, i.e.\ $\lim_{N\to\infty}\EE_{n\leq N}f(T^nx)\lio(n)=0$ for all $f\in C(X)$ and $x\in X$, Sarnak's conjecture (predicting the Liouville orthogonality of all zero topological entropy dynamical systems $(X,T)$) is equivalent to the strong LOMO property for the class of zero topological  entropy systems~\cite{Ab-Ku-Le-Ru}.

In order to see the relationship between~\eqref{unif1b} and the strong LOMO property~\eqref{smomo1}, we consider the homeomorphism $T\colon  G\times G/\Gamma\to G\times G/\Gamma$ given by
\beq\label{ours}
T\colon  (g,g'\Gamma)\mapsto(g,gg'\Gamma).\eeq
Then~\eqref{unif1b} can be read as   (remembering that $f\in C(G/\Gamma)$)
\beq\label{unif2}
\lim_{H\to\infty}\limsup_{M\to \infty}\EE_{m\leq M}\sup_{g,g'\in G}\Big|\EE_{h\leq H}\lio(m+h)f\circ T^{m+h}(g,g'\Gamma)\Big|=0\eeq
while~\eqref{smomo1} can be read as
\beq\label{smomo2}
\lim_{K\to\infty}\frac1{b_K}\sum_{k<K}\sup_{g,g'\in G}\Big|\sum_{b_k\leq n<b_{k+1}}\lio(n)f\circ T^n(g,g'\Gamma)\Big|=0\eeq
(see e.g.~\cite{Go-Le-Ru}). Now, the equivalence of~\eqref{unif2} and~\eqref{smomo2} follows from the following simple lemma.

\begin{Lemma}\label{l:ciagi}
Let $B$ be a normed space, and let $(z_n)\subset B$ be bounded. Then
\begin{align}\label{eq:normed1}
\lim_{H\to\infty}\limsup_{M\to\infty}\EE_{m\leq M}\Big\|\EE_{h\leq H}z_{m+h}\Big\|=0
\end{align}
if and only if for each increasing sequence $(b_k)\subset\N$ with $d(\{b_k\colon \,k\geq1\})=0$, we have
\begin{align}\label{eq:normed2}
\lim_{K\to\infty}\frac1{b_K}\sum_{k<K}\Big\|\sum_{b_k\leq n<b_{k+1}} z_n\Big\|=0.
\end{align}
An analogous result holds for logarithmic averages (with the logarithmic density $\delta$ in place of the density $d$).\end{Lemma}

It follows that the stronger local $t$-Fourier uniformity problem~\eqref{unif1b} is equivalent to the strong LOMO property of the homeomorphisms of the form~\eqref{ours}\footnote{The space $G\times G/\Gamma$ is locally compact  but need not be  compact. However, the problem of Liouville orthogonality (or of LOMO) of $T$ can still be studied here since we have a lot of {\em probability} $T$-invariant measures:  each such measure can be disintegrated over its projection on the first coordinate and the conditional (probability) measures are invariant under fiber nil-rotations. Since the latter are of zero entropy, in particular, it is  natural to consider $T$ as a zero entropy homeomorphism of $G\times G/\Gamma$. In particular, the Liouville orthogonality (or LOMO) can be studied for special continuous observables, for example, for the continuous functions depending only on the second coordinate.
Of course, the problem of non-compactness disappears if we take $C\subset G$ a {\em compact} subset and consider the relevant restriction of $T$. The problem is also irrelevant in abelian case, as we can simply consider $T$ as defined on $G/\Gamma\times G/\Gamma$: $T(g\Gamma,g'\Gamma)=(g\Gamma,gg'\Gamma)$.} and similarly with logarithmic averages (cf.~\cite{Ka-Ku-Le-Ru}). Hence, Sarnak's conjecture implies the local $t$-Fourier uniformity for each $t\geq1$.
Although Lemma~\ref{l:ciagi} is practically proved in~\cite{Ab-Le-Ru},~\cite{Ka-Ku-Le-Ru},~\cite{Ka-Le-Ul}, for the sake of completeness we will provide a proof in Appendix~\ref{App:A}.

\subsection{Strategy of the proofs}

\subsubsection{Proof of Theorems~\ref{p:jt1} and~\ref{thm:packingdimension}}

The proofs of Theorems~\ref{p:jt1} and~\ref{thm:packingdimension} are number-theoretic in nature. For the proof of Theorem~\ref{p:jt1}, we use the union bound, the second moment method and the density version of the two-point Chowla conjecture, proved in~\cite{Ta-Te}. When combined with the fact that $e((\alpha-\beta)h)=1+O(h|\alpha-\beta|)$, we prove the result for all sets $C$ that can be covered by $o(J)$ intervals of length $1/J$ as $J\to \infty$. A relatively short measure-theoretic argument shows that this property holds for all closed $C$ of Lebesgue measure $0$. For the proof of Theorem~\ref{thm:packingdimension}, we use largely the same strategy. The main difference is that for $t\geq 2$ we have $e((\alpha-\beta)h^t)=1+O(h^t|\alpha-\beta|)$, so we can only obtain the desired conclusion for sets $C$ that can be covered by $o(J^{1/t})$ intervals of length $1/J$ as $J\to \infty$.

\subsubsection{Proofs of Theorems~\ref{p:unif1} and~\ref{p:uniflog2A}}

The proofs of Theorems~\ref{p:unif1} and~\ref{p:uniflog2A} are dynamical; we prove these results by showing that the strong LOMO or logarithmic strong LOMO holds for certain systems, see~\cite{Fr-Ho},~\cite{Hu-Xu-Ye},~\cite{Ka-Le-Ra} and~\cite{Ta} for some other systems for which this property has been shown (see also Corollary~3.25 in~\cite{Fe-Ku-Le}). The scheme of proofs of our results is as follows:
\begin{itemize}
\item Take a class of topological systems for which we know that the (logarithmic) Liouville orthogonality holds.
\item Prove that in fact all the systems of this class satisfy the (logarithmic) strong LOMO property.
\item For Theorem~\ref{p:uniflog2A}, we use the class $\mathscr{A}_1$ of those zero entropy topological systems for which the set of ergodic measures is countable, relying on a celebrated theorem of Frantzikinakis and Host~\cite{Fr-Ho}. For Theorem~\ref{p:unif1}, we use the class $\mathscr{A}_2$ of systems whose all invariant measures yield  measure-theoretic systems which are $(q_n)$-rigid, relying on a result from~\cite{Ka-Le-Ra}. Lastly, we give an alternative proof of Corollary~\ref{p:uniflog1} by considering the class $\mathscr{A}_3$ of systems whose invariant measures yield measure-theoretic systems with singular spectra, relying on a dynamical interpretation of Tao's two-point logarithmic Chowla result~\cite{Ta}, given in Corollary 3.25 in the survey~\cite{Fe-Ku-Le}.
\end{itemize}

In Appendix~\ref{App:A}, we give the proof of Lemma~\ref{l:ciagi},
in Appendix~\ref{a:nowy} we show that results of a certain type can be extended from connected, simply connected Lie groups to more general Lie groups (which is needed for the proof of Theorem~\ref{p:uniflog2A}),
and in Appendix~\ref{App:B}, we show a general construction of subsets $C\subset\T$ of full Hausdorff dimension for which~\eqref{unif6} is satisfied.

\section{Proof of Theorem~\ref{p:jt1}}

\subsection{Lemmas for Theorem~\ref{p:jt1}}

We start with two-point Chowla, with Ces\`aro  averages at almost all scales.

\begin{Lemma}\label{le:two-point} There exists a set $\mathcal{M}\subset \mathbb{N}$ with $\delta(\mathcal{M})=1$ such that the following holds. For any integer $h\neq 0$,
\begin{align*}
\lim_{\substack{M\in \mathcal{M}\\M\to \infty}} \mathbb{E}_{m\leq M}\lio(m)\lio(m+h)=0.
\end{align*}
\end{Lemma}

\begin{proof}
This is~\cite[Corollary 1.13(ii)]{Ta-Te} with $g_1=g_2=\lio$.
\end{proof}

Lemma~\ref{le:two-point} quickly implies the following lemma about the frequency of large values of the sum $\sum_{h\leq H}\lio(m+h)e(\alpha h^s)$. We use $A\ll B$ to denote that $|A|\leq CB$ for some absolute constant $C$.

\begin{Lemma}\label{le:chebyshev} There exists a set $\mathcal{M}\subset \mathbb{N}$ with $\delta(\mathcal{M})=1$ such that for any $\varepsilon>0$ and $H\geq 1$,
\begin{align*}
\limsup_{\substack{M\in \mathcal{M}\\M\to \infty}}~\sup_{\|a\|_{\infty}\leq 1}~\frac{1}{M}\left|\left\{m\in [1,M]\cap \mathbb{N}\colon \,\, \left|\frac{1}{H}\sum_{h\leq H}\lio(m+h)a(h)\right|\geq \varepsilon\right\} \right|\ll \frac{\varepsilon^{-2}}{H},
\end{align*}
where the supremum is taken over all sequences $a\colon  \N\to\C$ with $\|a\|_\infty=\sup_{n\in\N}|a(n)|\leq 1$.
\end{Lemma}

\begin{proof} Let $\mathcal{M}$ be as in Lemma~\ref{le:two-point}.
Let $S_{a,M}$ be the set whose cardinality we are interested in. By Chebyshev's inequality and the fact that the number of $h_1,h_2\leq H$  with $h_2-h_1=h$ is $\max\{H-|h|,0\}$, we obtain
\begin{align*}
\frac{|S_{a, M}|}{M}&\leq (\varepsilon H)^{-2}\mathbb{E}_{m\leq M}\left|\sum_{h\leq H}\lio(m+h)a(h)\right|^2\\
&=(\varepsilon H)^{-2}\sum_{h_1,h_2\leq H}\mathbb{E}_{m\leq M}\lio(m+h_1)\lio(m+h_2)a(h_1)\overline{a(h_2)}\\
&\leq (\varepsilon H)^{-2}\sum_{h_1, h_2\leq H}\biggl|\mathbb{E}_{m\leq M}\lio(m)\lio(m+h_2-h_1)\biggr|+\frac{\varepsilon^{-2}H}{M}\\
&= (\varepsilon H)^{-2}\sum_{|h|\leq H}(H-|h|)\biggl|\mathbb{E}_{m\leq M}\lio(m)\lio(m+h)\biggr|+\frac{\varepsilon^{-2}H}{M}.
\end{align*}
The contribution of the terms $h\neq 0$ is, by Lemma~\ref{le:two-point}, $\ll H^{-100}$ (say) as soon as $M\in \mathcal{M}$ is large enough in terms of $H$. The contribution of the term $h=0$ in turn is $\ll \varepsilon^{-2}/H$. Hence, we conclude that
$$\limsup_{\substack{M\in \mathcal{M}\\M\to \infty}}~\sup_{\|a\|_\infty\leq 1}~|S_{\alpha, M}|/M\ll \varepsilon^{-2}/H,$$
as desired.
\end{proof}

\subsection{Reduction from zero measure to covering numbers}

In this subsection, we show that $C$ being closed and of measure zero implies a condition that is easier to work with in our proof of Theorem~\ref{p:jt1}. In fact, we prove this in a slightly more general form for sets that are allowed to have positive measure\footnote{We thank the referee for pointing out this generalization.}. 
For a set $C\subset \mathbb{R}$, let $N_r(C)$ be the least number of closed intervals of length $r$ whose union covers $C$.

\begin{Lemma}\label{le:closed}
Let $C\subset [0,1]$ be a closed set. Then
\begin{align*}
\limsup_{r\to 0}rN_r(C)\leq \Leb(C).
\end{align*}
\end{Lemma}

\begin{proof} This is somewhat similar to~\cite[Lemma 6.6]{Ta0}.
Let $\varepsilon>0$. By the definition of the Lebesgue measure, there is an open set $U$ such that $C\subset U$ and $\Leb(U)\leq \Leb(C)+\varepsilon$. By compactness, we can assume that $U$ is a union of finitely many intervals $I_1,\ldots, I_K$ for some natural number $K$. Removing any overlapping parts of the intervals $I_j$ (and possibly adding some singleton intervals), we may also assume that the $I_j$ are disjoint.  Note that each $I_j$ satisfies $\lim_{r\to 0}rN_r(I_j)=\Leb(I_j)$. Hence,
\begin{align*}
\limsup_{r\to 0}rN_r(C)\leq \sum_{j=1}^{K}\limsup_{r\to 0}rN_r(I_j)=\sum_{j=1}^{K}\Leb(I_j)\leq \Leb(C)+\varepsilon.    
\end{align*}
Letting $\varepsilon\to 0$, the claim follows.
\end{proof}




\subsection{Proof of Theorem~\ref{p:jt1b}}

We will in fact prove the following more general theorem from which Theorem~\ref{p:jt1} is an immediate consequence.

\begin{Th}\label{p:jt1b} There exists a set $\mathcal{M}\subset \mathbb{N}$ of logarithmic density $1$ such that the following holds.
Let $C\subset\T$ be any closed set. Then we have
\begin{align*}
\lim_{H\to \infty}\limsup_{\substack{M\in \mathcal{M}\\M\to \infty}}\mathbb{E}_{m\leq M}\sup_{\alpha\in C}\Big|\EE_{m\leq h< m+H}\lio(h)e(h\alpha)\Big|\ll \Leb(C)^{1/4}.
\end{align*}
\end{Th}

\begin{proof} Let $\mathcal{M}\subset \mathbb{N}$ with $\delta(\mathcal{M})=1$ be as in Lemma~\ref{le:chebyshev}. 
Write $\varepsilon=\Leb(C)\in [0,1]$. Then by Lemma~\ref{le:closed} for every $H$ large enough in terms of $\varepsilon$ there exist some  $J=J_H\leq 2\varepsilon^{3/4} H$ and $\alpha_1,\ldots, \alpha_J\in \mathbb{R}$ (depending on $H$) such that
\begin{align*}
C\subset \bigcup_{j\leq J}\left[\alpha_j,\alpha_j+\frac{\varepsilon^{1/4}}{H}\right].
\end{align*}
Using $e(\beta)=1+O(|\beta|)$, it follows that
\begin{align*}
\sup_{\alpha\in C}\left|\sum_{h\leq H}\lio(m+h)e(\alpha h)\right|\leq \max_{j\leq J} \left|\sum_{h\leq H}\lio(m+h)e(\alpha_jh)\right|+O(\varepsilon^{1/4} H).
\end{align*}
Dividing both sides by $H$ and considering the values of
\[
\sup_{\alpha \in C} \left|\sum_{h\leq H}\lio(m+h)e(\alpha h)\right|
\]
smaller than $K_0\vep^{1/4} H$ (and greater than this number), for $K_0$ a large absolute constant and for $\varepsilon>0$ small enough, we obtain
\begin{align*}
\limsup_{\substack{M\in \mathcal{M}\\M\to \infty}}\mathbb{E}_{m\leq M} & \sup_{\alpha\in C} \Big|\frac{1}{H}\sum_{h\leq H}\lio(m+h)e(\alpha h)\Big|\\
\leq& K_0\varepsilon^{1/4}+\limsup_{\substack{M\in \mathcal{M}\\M\to \infty}}\frac{1}{M}\left|\left\{m\leq M\colon  \,\, \max_{j\leq J}\left|\sum_{h\leq H}\lio(m+h)e(\alpha_jh)\right|\geq \varepsilon^{1/4} H\right\}\right|\\
\leq& K_0\varepsilon^{1/4}+\sum_{j\leq J} \limsup_{\substack{M\in \mathcal{M}\\M\to \infty}}\frac{1}{M}\left|\left\{m\leq M\colon  \,\, \left|\sum_{h\leq H}\lio(m+h)e(\alpha_jh)\right|\geq \varepsilon^{1/4} H\right\}\right|.
\end{align*}
By Lemma~\ref{le:chebyshev}, this is
\begin{align*}
\leq K_0\varepsilon^{1/4}+O(\varepsilon^{-1/2}J/H)\ll \varepsilon^{1/4},
\end{align*}
recalling that $J\leq 2 \varepsilon^{3/4} H$. This completes the proof.    
\end{proof}

\section{Proof of Theorem~\ref{thm:packingdimension}}

\begin{proof}[Proof of Theorem~\ref{thm:packingdimension}]
Let $\mathcal{M}\subset \mathbb{N}$ with $\delta(\mathcal{M})=1$ be as guaranteed by
Lemma~\ref{le:chebyshev}.
We must show that for any $\varepsilon>0$ and for any $C\subset [0,1]$ of upper box-counting dimension $<1/t$, we have
\begin{align}\label{eq1b}
\limsup_{H\to \infty}\limsup_{\substack{M\in \mathcal{M}\\M\to \infty}}\mathbb{E}_{m\leq M}\sup_{\alpha\in C}\bigg|\frac{1}{H}\sum_{h\leq H}\lio(m+h)e(\alpha h^t)\bigg|\leq \varepsilon.
\end{align}

Since $C$ has upper box-counting dimension $<1/t$, for any large enough $H$ there is some $J\leq \varepsilon^{4}H$ and some $\alpha_1,\ldots, \alpha_J\in \mathbb{R}$ such that we have
\begin{align*}
C\subset \bigcup_{j\leq J}\left[\alpha_j,\alpha_j+\frac{\varepsilon^2}{H^{t}}\right].
\end{align*}
Note that since $e(\beta)=1+O(|\beta|)$ we have
\begin{align*}
\sup_{\alpha\in C}\left|\sum_{h\leq H}\lio(m+h)e(\alpha h^t)\right|\leq \max_{j\leq J} \left|\sum_{h\leq H}\lio(m+h)e(\alpha_jh^t)\right|+O(\varepsilon^2 H).
\end{align*}

If we let
\begin{align*}
S_{\alpha}=\left\{m\in\mathbb{N}\colon  \,\, \left|\frac{1}{H}\sum_{h\leq H}\lio(m+h)e(\alpha h^t)\right|\geq \frac{\varepsilon}{2}\right\}
\end{align*}
then, for $\varepsilon>0$ small enough and $H$ sufficiently large, by the union bound, we have
\begin{align*}
\limsup_{\substack{M\in \mathcal{M}\\M\to \infty}}\mathbb{E}_{m\leq M}\sup_{\alpha\in C} & \left|\frac{1}{H}\sum_{h\leq H}\lio(m+h)e(\alpha h^t)\right|\\
\leq& \limsup_{\substack{M\in \mathcal{M}\\M\to \infty}}\mathbb{E}_{m\leq M}\max_{j\leq J}\left|\frac{1}{H}\sum_{h\leq H}\lio(m+h)e(\alpha_j h^t)\right|\,+\,O(\varepsilon^2)\\
\leq& \frac{\varepsilon}{2}+\limsup_{\substack{M\in \mathcal{M}\\M\to \infty}}
\mathbb{E}_{m\leq M} \max_{j\leq J} \raz_{S_{\alpha_j}}(m)\,+\,O(\varepsilon^2 H)
\\
\leq& \frac{\varepsilon}{2}+\sum_{j\leq J}\biggl(\limsup_{\substack{M\in \mathcal{M}\\M\to \infty}}
\mathbb{E}_{m\leq M}  \raz_{S_{\alpha_j}}(m)\biggr).
\end{align*}
Applying Lemma~\ref{le:chebyshev} with $a(h)=e(\alpha_j h^t)$, this is
\begin{align*}
\leq \frac{\varepsilon}{2}+O\left(\varepsilon^{-2}\frac{J}{H}\right)\leq \varepsilon  \end{align*}
for $\varepsilon>0$ small enough, recalling that $J\leq \varepsilon^{4} H$.
\end{proof}

\section{Optimality of results}\label{sec:implication}

In this section, we prove two theorems showing that our results on the logarithmic local $1$-Fourier uniformity problem cannot be extended much without settling the full conjecture. One of them is Theorem~\ref{th:implication}, stated in the introduction. The other one is the following implication.

\begin{Th}\label{th:assum} Suppose that there is a measurable set $C\subset \T$ of positive Lebesgue measure such that for all Dirichlet characters $\chi$ we have
\begin{align}\label{eq:assum}
\limsup_{H\to \infty}\limsup_{M\to \infty}\mathbb{E}_{m\leq M}^{\log} \sup_{\alpha\in C}|\mathbb{E}_{h\leq H}\lio(m+h)\chi(m+h)e(h\alpha)|=0.   
\end{align}
Then the logarithmic local $1$-Fourier uniformity conjecture holds. 
\end{Th}

Morally speaking, the assumption~\eqref{eq:assum} is not much stronger than the case $\chi=1$ (for example, the proof of Corollary~\ref{p:uniflog1} carries through with a character twist), although we cannot prove a rigorous implication from the case $\chi=1$ to the general case.

\subsection{Proof of Theorem~\ref{th:assum}}

In this subsection, we prove Theorem~\ref{th:assum}. 

\begin{proof}[Proof of Theorem~\ref{th:assum}]
We first claim that under the assumption of the theorem we have
\begin{align}\label{eq:lincom}
\limsup_{H\to \infty}\limsup_{M\to \infty}\mathbb{E}_{m\leq M}^{\log}\sup_{\alpha\in C}|\mathbb{E}_{h\leq H}\lio(m+h)e(h\alpha)1_{h\equiv a\pmod r}|=0    
\end{align}
for any integers $a,r\geq 1$. Indeed,  we have
\begin{align*}
&\limsup_{H\to \infty}\limsup_{M\to \infty}\mathbb{E}_{m\leq M}^{\log}\sup_{\alpha\in C}|\mathbb{E}_{h\leq H}\lio(m+h)e(h\alpha)1_{h\equiv a\pmod r}|\\
&= \limsup_{H\to \infty}\limsup_{M\to \infty}\mathbb{E}_{m\leq M}^{\log}r1_{m\equiv -a+1\pmod r}
\sup_{\alpha\in C}|\mathbb{E}_{h\leq H}\lio(m+h)e(h\alpha)1_{h\equiv a\pmod r}|\\
&=\limsup_{H\to \infty}\limsup_{M\to \infty}\mathbb{E}_{m\leq M}^{\log}r1_{m\equiv -a+1\pmod r}
\sup_{\alpha\in C}|\mathbb{E}_{h\leq H}\lio(m+h)e(h\alpha)1_{m+h\equiv 1\pmod r}|
\\
&=0,
\end{align*}
since $1_{m+h\equiv 1\pmod r}$ is a finite linear combination of Dirichlet characters $\pmod r$ evaluated at $m+h$. Taking linear combinations of~\eqref{eq:lincom}, we now conclude that
\begin{align}\label{eq:rat}
\limsup_{H\to \infty}\limsup_{M\to \infty}\mathbb{E}_{m\leq M}^{\log}\sup_{\alpha\in C}|\mathbb{E}_{h\leq H}\lio(m+h)e(h(\alpha+\beta))|=0    
\end{align}
for every rational number $\beta$. 

Let $\varepsilon\in (0,1)$. Since $C$ is measurable and of positive measure, by Lebesgue's density theorem there exist integers $0\leq a\leq q-1$ such that for the interval $I=[a/q,(a+1)/q]$ we have $\Leb(C\cap I)\geq (1-\varepsilon) \Leb(I)$. Therefore we have
\begin{align}\label{eq:union}
\Leb\Bigg([0,1]\setminus \bigcup_{b=0}^{q-1}(C+\frac{b}{q})\Bigg)\leq \varepsilon.      
\end{align}
We now  estimate
\begin{align*}
&\limsup_{H\to \infty}\limsup_{M\to \infty}\mathbb{E}_{m\leq M}^{\log}\sup_{\alpha\in [0,1]}|\mathbb{E}_{h\leq H}\lio(m+h)e(h\alpha)|\\
&\leq \limsup_{H\to \infty}\limsup_{M\to \infty}\mathbb{E}_{m\leq M}^{\log}\sup_{\alpha\in \bigcup_{b=0}^{q-1} (C+b/q)}|\mathbb{E}_{h\leq H}\lio(m+h)e(h\alpha)|\\
&\quad +\limsup_{H\to \infty}\limsup_{M\to \infty}\mathbb{E}_{m\leq M}^{\log}\sup_{\alpha\in [0,1]\setminus \bigcup_{b=0}^{q-1} (C+b/q)}|\mathbb{E}_{h\leq H}\lio(m+h)e(h\alpha)|\\
&\leq \sum_{b=0}^{q-1}\limsup_{H\to \infty}\limsup_{M\to \infty}\mathbb{E}_{m\leq M}^{\log}\sup_{\alpha\in C}|\mathbb{E}_{h\leq H}\lio(m+h)e(h(\alpha+\frac{b}{q}))|\\
&\quad +\limsup_{H\to \infty}\limsup_{M\to \infty}\mathbb{E}_{m\leq M}^{\log}\sup_{\alpha\in [0,1]\setminus \bigcup_{b=1}^{q-1} (C+b/q)}|\mathbb{E}_{h\leq H}\lio(m+h)e(h\alpha)|.
\end{align*}
Each of the summands in the sum over $b$ is $0$ by~\eqref{eq:rat}. Moreover, the expression on the last line is by Theorem~\ref{p:jt1b} and~\eqref{eq:union},
\begin{align*}
\ll \Leb\Bigg([0,1]\setminus \bigcup_{b=1}^{q-1} (C+b/q)\Bigg)^{1/4}\leq \varepsilon^{1/4}.    
\end{align*}
Letting $\varepsilon\to 0$, the claim follows. 
\end{proof}

\subsection{Proof of Theorem~\ref{th:implication}}

In this subsection, we prove Theorem~\ref{th:implication}.

We begin with the following lemma about the discrepancy of the sequence $p\alpha\pmod 1$, where $p$ runs over primes. As usual, we define the discrepancy of a sequence $(x_k)_{k=1}^{K}\subset [0,1]$ by
\begin{align*}
\sup_{\substack{I\subset [0,1]\\I\textnormal{ interval}}}\left|\frac{1}{K}|\{k\leq K\colon x_k\in I\}|-\Leb(I)\right|.
\end{align*}

\begin{Lemma}\label{le:discrepancy} Let $\varepsilon\in (0,1)$ and $P\geq \varepsilon^{-10}$. Let $\alpha \in \mathbb{R}$. Then the discrepancy of the sequence $(p\alpha\pmod 1)_{p\leq P}$ is $\leq \varepsilon$ unless there exists an integer $1\leq \ell\ll \varepsilon^{-10}$ such that $\|\ell \alpha\|\ll \varepsilon^{-10}/P$.
\end{Lemma}

\begin{proof}
By the Erd\H{o}s--Tur\'an inequality, if the discrepancy of $(p\alpha\pmod 1)_{p\leq P}$ is $> \varepsilon$, then for some integer $1\leq k\ll \varepsilon^{-2}$ we have 
\begin{align}\label{eq:ET}
\left|\mathbb{E}_{p\leq P}e(kp\alpha)\right|\gg \varepsilon^2.
\end{align}
Let $Q=\varepsilon^5P/(\log P)^{10}$. Then, 
Dirichlet's approximation theorem tells us that for some integers $1\leq \ell\leq Q$ and $a$ we have $|\alpha-a/\ell|\leq 1/(\ell Q)\leq 1/\ell^2$. Using a standard estimate for exponential sums of the primes~\cite[Theorem 13.6]{Iwaniec-Kowalski}\footnote{This is stated with the von Mangoldt weight, but a similar estimate holds for the unweighted prime sum.}, this implies that $\ell\ll \varepsilon^{-5}(\log P)^{10}$. But then
$$\Big|\alpha-\frac{a}{\ell}\Big|\leq \frac{\varepsilon^{-5}(\log P)^{10}}{P}.$$
This gives the desired claim if $\varepsilon\leq (\log P)^{-2}$. Suppose then that $\varepsilon>(\log P)^{-2}$.  Then 
we have $\alpha=a/\ell+M/P$ for some $(\log P)^{10}\leq |M|\ll (\log P)^{20}$. But now, splitting into short intervals and arithmetic progressions, we have
\begin{align*}
&\mathbb{E}_{p\leq P}e\left(k\left(\frac{a}{\ell}+\frac{M}{P}\right)p\right)\\
&=\sum_{\substack{1\leq j\leq \ell\\(j,\ell)=1}}e\left(\frac{akj}{\ell}\right)\frac{1}{P}\int_{1}^{P}\mathbb{E}_{t< p\leq t+\varepsilon^{2}M^{-1}P}1_{p\equiv j\pmod{\ell}}e\left(\frac{M}{P}t\right) d t+O(\varepsilon^2)\\
&\ll \frac{1}{\varphi(\ell)}\cdot \frac{\varepsilon^{-2}}{M}+O(\varepsilon^2),
\end{align*}
where for the last line we used the Siegel--Walfisz theorem, an evaluation of the Ramanujan sum, and the bounds $\varepsilon>(\log P)^{-2}$, $|M|\ll (\log P)^{20}$. 
Comparing with~\eqref{eq:ET}, we conclude that $\ell\ll \varepsilon^{-10}$ and $M\ll \varepsilon^{-10}$, so the claim follows. 
\end{proof}

\begin{proof}[Proof of Theorem~\ref{th:implication}]
Since $C$ has non-empty interior, there exist real numbers $0<a<b<1$ such that $[a,b]\subset C$. Hence by assumption we have
\begin{align}\label{eq:assumption}
\limsup_{H\to \infty}\limsup_{X\to \infty}\mathbb{E}_{x\leq X}^{\log}\sup_{\alpha \in [a,b]}|\mathbb{E}_{x<n\leq x+H}\lio(n)e(\alpha n)|=0.
\end{align}

Our task now is to show that for any function $\alpha\colon \mathbb{R}\to [0,1]$ we have
\begin{align}\label{eq:alphaclaim}
\limsup_{H\to \infty}\limsup_{X\to \infty}\mathbb{E}_{x\leq X}^{\log}|\mathbb{E}_{x<n\leq x+H}\lio(n)e(\alpha(x)n)|=0.
\end{align}
Let $\varepsilon>0$ be small, and let $P$ be large enough in terms of $\varepsilon$. We first claim that
\begin{align}\label{eq:Elliott}\begin{split}
&\limsup_{H\to \infty}\limsup_{X\to \infty}\mathbb{E}_{x\leq X}^{\log}\mathbb{E}_{p\leq P}^{\log}|\mathbb{E}_{x<n\leq x+H}\lio(n)e(\alpha(x)n)\\
&\quad +\mathbb{E}_{x/p<m\leq (x+H)/p}\lio(m)e(p\alpha(x)m)|\leq \varepsilon.
\end{split}
\end{align}
Indeed, by the complete multiplicativity of $\lio$ and Elliott's inequality~\cite[Proposition 2.5]{Ma-Ra-Ta} (with $f(n)=\lio(n)e(\alpha(x)n)$ and $\delta=(\log \log P)^{-1/10}$), for any $x\geq H\geq P\geq 3$ we have
\begin{align*}
-\mathbb{E}_{x/p< m\leq (x+H)/p}\lio(m)  e(\alpha(x)pm)
=\mathbb{E}_{x< n\leq x+H}\lio(n)e(\alpha(x)n)+O((\log \log P)^{-1/10})
\end{align*}
for all primes $p\leq P$ outside an exceptional set of $p$ with logarithmic sum $\ll (\log \log P)^{1/5}$. Now, averaging over $p\leq P$, we get
\begin{align*}
\mathbb{E}_{p\leq P}^{\log}|-\mathbb{E}_{x/p< m\leq (x+H)/p}\lio(m)e(\alpha(x)pm)-  \mathbb{E}_{x<n\leq x+H}  \lio(n) e(\alpha(x)n)|
\ll (\log \log P)^{-1/10}.
\end{align*}
Further taking logarithmic averages over $x\leq X$ and taking limsups the claim~\eqref{eq:Elliott} follows (since $P$ is large enough in terms of $\varepsilon$).

Now we have~\eqref{eq:Elliott}. Restricting the $p$ average in that estimate, we obtain
\begin{align}\label{eq:Elliott2}\begin{split}
\limsup_{H\to \infty}\limsup_{X\to \infty}\mathbb{E}_{x\leq X}^{\log}\mathbb{E}_{p\leq P}^{\log}1_{p\alpha(x)\in [a,b]\pmod 1} & \Big|\mathbb{E}_{x<n\leq x+H}\lio(n)e(\alpha(x)n)\\
&\quad +\mathbb{E}_{x/p<m\leq (x+H)/p}\lio(m)e(p\alpha(x)m)\Big|\leq \varepsilon.
\end{split}
\end{align}
By the dilation invariance of logarithmic averages and~\eqref{eq:assumption}, we can bound
\begin{align*}
&\limsup_{H\to \infty}\limsup_{X\to \infty}\mathbb{E}_{x\leq X}^{\log}\mathbb{E}_{p\leq P}^{\log}1_{p\alpha(x)\in [a,b]\pmod 1}|\mathbb{E}_{x/p<m\leq (x+H)/p}\lio(m)e(p\alpha(x)m)|\\
&=\limsup_{H\to \infty}\limsup_{X\to \infty}
\mathbb{E}_{y\leq X}^{\log}\mathbb{E}_{p\leq P}^{\log}1_{p\alpha(py)\in [a,b]\pmod 1}|\mathbb{E}_{y<m\leq y+H/p}\lio(m)e(p\alpha(py)m)|\\
&\leq \mathbb{E}_{p\leq P}^{\log}\limsup_{H\to \infty}\limsup_{X\to \infty}
\mathbb{E}_{y\leq X}^{\log}1_{p\alpha(py)\in [a,b]\pmod 1}|\mathbb{E}_{y<m\leq y+H/p}\lio(m)e(p\alpha(py)m)|
\\
&=0.
\end{align*}
Hence, by the triangle inequality,~\eqref{eq:Elliott2} implies
\begin{align}\label{eq:alphabound}
\limsup_{H\to \infty}\limsup_{X\to \infty}\mathbb{E}_{x\leq X}^{\log}\mathbb{E}_{p\leq P}^{\log}1_{p\alpha(x)\in [a,b]\pmod 1}|\mathbb{E}_{x<n\leq x+H}\lio(n)e(\alpha(x)n)|\leq \varepsilon.
\end{align}

 Introduce the sets 
\begin{align*}
\mathcal{X}_1&=\Big\{x\in [1,X]\colon \sum_{p\leq P}\frac{1_{p\alpha(x)\in [a,b]\pmod 1}}{p}\geq \varepsilon^{1/2} \log \log P+1\Big\},\\
\mathcal{X}_2&=[1,X]\setminus \mathcal{X}_1.
\end{align*}
Then it suffices to show that for $j\in \{1,2\}$ we have
\begin{align}\label{eq:jclaim}
\limsup_{H\to \infty}\limsup_{X\to \infty}\mathbb{E}_{x\leq X}^{\log}1_{\mathcal{X}_j}(x)|\mathbb{E}_{x<n\leq x+H}\lio(n)e(\alpha(x)n)|\leq \varepsilon^{1/2},
\end{align}
as  letting $\varepsilon\to 0$ the desired claim~\eqref{eq:alphaclaim} follows. 

Using~\eqref{eq:alphabound}, we have 
\begin{align}\label{eq:claims}\begin{split}
&\limsup_{H\to \infty}\limsup_{X\to \infty}\mathbb{E}_{x\leq X}^{\log}1_{\mathcal{X}_1}(x)|\mathbb{E}_{x<n\leq x+H}\lio(n)e(\alpha(x)n)|\\
&\leq \varepsilon^{-1/2}\limsup_{H\to \infty}\limsup_{X\to \infty}\mathbb{E}_{x\leq X}^{\log}1_{\mathcal{X}_1}(x)\mathbb{E}_{p\leq P}^{\log}1_{p\alpha(x)\in [a,b]\pmod 1}|\mathbb{E}_{x<n\leq x+H}\lio(n)e(\alpha(x)n)|\\
&\leq \varepsilon^{1/2}.
\end{split}
\end{align}
Hence, what remains to be shown is that
\begin{align}\label{eq:last}
\limsup_{H\to \infty}\limsup_{X\to \infty}\mathbb{E}_{x\leq X}^{\log}1_{\mathcal{X}_2}(x)|\mathbb{E}_{x<n\leq x+H}\lio(n)e(\alpha(x)n)|\leq \varepsilon^{1/2}.
\end{align}

Note that for $x\in \mathcal{X}_2$ there exists $P'\in [\log P, P/2]$ such that $p\alpha(x)\in [a,b]\pmod 1$ holds for $< \frac{b-a}{2}\frac{P'}{\log P'}$ primes $p\in [P',2P']$. By Lemma~\ref{le:discrepancy}, we conclude that there exists an absolute constant $C_0\geq 1$ such that, for each $x\in \mathcal{X}_2$, there is an integer $1\leq \ell\leq C_0\varepsilon^{-10}$ for which $\|\ell \alpha(x)\|\leq C_0 \varepsilon^{-10}/(\log P)$. Since $P$ is large enough in terms of $\varepsilon$, by splitting the sums of length $H$ into sums of length $(\log P)^{1/2}$ and applying the triangle inequality, we see that 
\begin{align*}
\limsup_{H\to \infty}\limsup_{X\to \infty} & \mathbb{E}_{x\leq X}^{\log}1_{\mathcal{X}_2}(x)|\mathbb{E}_{x<n\leq x+H} \lio(n)e(\alpha(x)n)|\\
&\leq \limsup_{X\to \infty}\mathbb{E}_{x\leq X}^{\log}1_{\mathcal{X}_2}(x)|\mathbb{E}_{x<n\leq x+(\log P)^{1/2}}\lio(n)e(\alpha(x)n)|+O(\varepsilon^2)\\
&\leq \sum_{1\leq k\leq \ell \leq C_0\varepsilon^{-10}}\limsup_{X\to \infty}\mathbb{E}_{x\leq X}^{\log}\Big|\mathbb{E}_{x<n\leq x+(\log P)^{1/2}}\lio(n)e\Big(\frac{kn}{\ell}\Big)\Big|+O(\varepsilon^2).
\end{align*}
But by the Matom\"aki--Radziwi\l{}\l{}--Tao estimate~\cite[Theorem 1.3]{MRT} for short exponential sums of the Liouville function, and the assumption that $P$ is large enough in terms of $\varepsilon$, this is $\leq \varepsilon$ if  $\varepsilon>0$ is small enough. This gives~\eqref{eq:last}, completing the proof.
\end{proof}

\section{An alternative proof of Corollary~\ref{p:uniflog1}}
Let us recall some basic notions of topological dynamics and ergodic theory.
Let $T$ be a homeomorphism of a compact metric space $X$. By $M(X)$ we denote the space of all (Borel) probability measures on $X$. $M(X)$ endowed with the weak-$\ast$ topology is compact and the set $M(X,T)$ of $T$-invariant Borel probability measures on $X$ is a non-empty and closed subset of it.

Any member $\mu\in M(X,T)$ yields a measure-theoretic system $(X,\mu,T)$. Moreover, $T$ induces a unitary operator $U_T(f)=Tf\coloneqq f\circ T$ on $L^2(X,\mu)$. Then, the Herglotz theorem implies that each $f$ determines a unique (Borel) positive  finite measure $\sigma_f$ on $\bs^1$ whose Fourier transform is given by
$$
\widehat{\sigma}_f(n)\coloneqq \int_{\bs^1}z^n\,d\sigma_f(z)=\int_X T^nf\cdot\ov{f}\,d\mu\quad \text{for all}\quad n\in\Z.$$
Among spectral measures there are maximal ones (with respect to the absolute continuity relation). Those maximal elements are called measures of maximal spectral type which are all mutually absolutely continuous with respect to one another.
Recall that $(X,\mu,T)$ is rigid along $(q_n)$  if $T^{q_n}f\to f$ in $L^2(X,\mu)$ for each $f\in L^2(X,\mu)$. As $\widehat{\sigma}_f(q_n)\to1$ for each $f\in L^2(X,\mu)$ with $\|f\|=1$, by the Riemann--Lebesgue lemma, it follows that rigid systems have singular maximal spectral type.

If $(X',\mu',T')$ is another measure-theoretic system, then by a joining of it with $(X,\mu,T)$ we mean an element of $\rho\in M(X'\times X,T'\times T)$ such that its projections on $X'$ and $X$ are $\mu'$ and $\mu$, respectively. Clearly, $(X'\times X,\rho,T'\times T)$ is a measure-theoretic dynamical system.

\subsection{Lemmas}
\begin{Lemma}\label{l:zid}
Assume that $(X,\mu,T)$ and $(Y,\nu, \textnormal{Id})$ are two dynamical systems. Let $\rho$ be a joining of $T$ and $\textnormal{Id}$. Then the maximal spectral types of $T$ and of $T\times \textnormal{Id}$ are the same.\end{Lemma}
\begin{proof} Consider $F=f\ot g$ with $|g|=1$. We have
$$
\int F\circ (T\times \textnormal{Id})^n \ov{F}\,d\rho=
\int f(T^nx)\ov{f(x)}\cdot|g(y)|^2\,d\rho(x,y)=\int f(T^nx)\ov{f(x)}\,d\mu(x),$$
so the spectral measure of $F$ is the same as that of $f$.\end{proof}

\begin{Lemma}\label{l:supinside2}
Let $G$ be a compact abelian group and $C$ a closed subset of it. Let $T\colon  C\times G\to C\times G$ be given by $T(x,g)=(x,g+x)$. Then $T$ is a homeomorphism of $C\times G$ and for each $\rho\in M(C\times G,T)$ the maximal spectral type  of the unitary operator $U_T$ acting on $L^2(C\times G,\rho)$ is equal to\footnote{By $\chi_\ast(\sigma)$ we denote the image of $\sigma$ via the map $\chi$.}
$$
\sigma_T=\sum_{\chi\in \widehat{G}} a_\chi \chi_\ast(\sigma),$$
where $a_\chi>0$, $\sum_{\chi\in\widehat{G}}a_\chi<+\infty$ and $\sigma=\pi_\ast(\rho)$ with $\pi(x,g)=x$.
\end{Lemma}
\begin{proof} Let $F(x,g)=f(x)\chi(g)$ with $\chi\in \widehat{G}$. Then
$$
\int F(T^n(x,g))\ov{F(x,g)}\,d\rho(x,g)=\int \chi(nx)|f(x)|^2\,d\rho(x,g)
$$
$$
=\int (\chi(x))^n|f(x)|^2\,d\rho(x,g)=\int (\chi(x))^n|f(x)|^2\,d\sigma(x)=
\int z^n\, d\chi_\ast(|f(x)|^2\sigma),$$
so
$$
\sigma_F=\chi_\ast(|f(x)|^2\sigma)\ll \chi_\ast(\sigma)=\sigma_{1\otimes \chi}$$
and the result follows.
\end{proof}

\begin{Lemma} \label{l:supinside3}
Let $(X,T)$ be a topological system in which for all $\nu\in M(X,T)$ the corresponding measure-theoretic dynamical system $(X,\nu, T)$ has singular spectral type.
Then $(X,T)$ satisfies the  logarithmic strong LOMO property.
\end{Lemma}
\begin{proof}
Let $(x_k)\subset X$ and let $(b_k)\subset\N$ satisfy
$\delta(\{b_k\colon \,k\geq1\})=0$.
We want to show that
\beq\label{olala}
\frac1{\log b_K}\sum_{k<K}\Big|\sum_{b_k\leq n<b_{k+1}}\frac1nf(T^nx_k)\lio(n)\Big|\to0.\eeq
Consider the space $X\times Y$, with $Y=\{e^{2\pi ij/3}\colon \, j=0,1,2\}$ (on $Y$ we consider the action of identity). Let $((x_k,a_k))_{k\geq1}\subset X\times Y$  with $a_k$ to be specified shortly. Set $\tilde{f}(x,\eta)=f(x)\eta$ for $x\in X$ and $\eta\in Y$. Then
\begin{align*}
&\frac1{\log b_K}\sum_{k<K}\sum_{b_k\leq n<b_{k+1}}\frac{1}{n}\tilde{f}(T^n x_k,a_k)\lio(n)\\
=&\frac1{\log b_K}\sum_{k<K}a_k\sum_{b_k\leq n<b_{k+1}}\frac{1}{n}f(T^nx_k)\lio(n)
\end{align*}
and we can select $(a_k)\subset Y$ so that the values $a_k\sum_{b_k\leq n<b_{k+1}}\frac1nf(T^nx_k)$ lie in a fixed convex cone (in $\C$) of angle $<\pi$.
Let $S$ denote the left-shift on the symbolic shift-space $\{-1,1\}^\Z$ and let $X_{\lio}$ denote the orbit closure of $\lio$ under $S$ (where we view $\lio$ as an element of $\{-1,1\}^\Z$ by extending it to $\Z$ in an arbitrary way using $\pm1$).
In view of~\cite[Lemma~18]{Ab-Ku-Le-Ru},~\eqref{olala} is now equivalent to
\begin{align}\label{olala1}\begin{split}&\lim_{K\to\infty}\frac1{\log b_K}\sum_{k<K}a_k\sum_{b_k\leq n<b_{k+1}}\frac1nf(T^nx_k)\pi_0(S^n\lio)=0\\
=&\lim_{K\to\infty}\frac1{\log b_K}\sum_{k<K}\Big|\sum_{b_k\leq n<b_{k+1}}\frac1nf(T^nx_k)\pi_0(S^n\lio)\Big|.
\end{split}\end{align}

To compute the limit of the left-hand side above, consider the sequence
$$
\left(\frac1{\log b_K}\sum_{k<K}\sum_{b_k\leq n<b_{k+1}}\frac1n\delta_{(T\times \textnormal{Id}\times S)^n(x_k,a_k,\lio)}\right)_{K\geq1}\subset M(X\times Y\times X_{\lio}).$$
By passing to a subsequence if necessary, we can assume that this sequence converges to a measure $\rho$ which, by the zero logarithmic density of $(b_k)$, must be $T\times \textnormal{Id}\times S$-invariant. So, it is a joining of $\nu\in M(X,T)$, $\nu'\in M(Y, \textnormal{Id})$ and a Furstenberg system  $\kappa$ of $\lio$ \cite{Fr-Ho}; the latter is true because $\kappa\in M(X_{\lio},S)$, where $\kappa$ satisfies
$$
\kappa=\lim_{K\to\infty}\frac1{\log b_K}\sum_{k<K}\sum_{b_k\leq n<b_{k+1}}\frac1n\delta_{S^n\lio}=\lim_{K\to\infty}\frac1{\log b_K}\sum_{ n<b_{K}}\frac1n\delta_{S^n\lio}.$$

Hence, the limit of the left-hand side in~\eqref{olala1} is $\int \tilde{f}\ot\pi_0\,d\rho$. Because of Lemma~\ref{le:two-point}, the spectral measure $\sigma_{\pi_0}$ for $\pi_0$ understood as an element of $L^2(\rho)$ (this spectral measure is precisely the same as the spectral measure of $\pi_0$ when viewed as an element of $L^2(X_{\lio},\kappa)$ since $\rho$ is a joining) is the Lebesgue measure on the circle, see~\cite{Fe-Ku-Le}. On the other hand, by Lemma~\ref{l:zid} and our assumption that any measure in $M(X,T)$ yields a dynamical system of singular spectral type, the spectral measure $\sigma_{\tilde{f}}$  of $\tilde{f}\in L^2(\rho)$ is singular. Therefore, $\tilde{f}$ and $\pi_0$ are orthogonal and hence~\eqref{olala1} holds.
\end{proof}

\subsection{Proof of Corollary~\ref{p:uniflog1}}
We apply the above to $T(x,y)=(x,x+y)$ on $C\times \T$. In view of Lemma~\ref{l:supinside2}, for each invariant measure for $T$ the maximal spectral type of the measure-theoretic system corresponding to the measure is singular (as $C$ has Lebesgue measure zero, each measure on it must be singular with respect to the Lebesgue measure). It follows from Lemma~\ref{l:supinside3} that $(C\times \T,T)$ satisfies the logarithmic strong LOMO property, which we apply to $f(x,y)=e^{2\pi iy}$. Finally, use Lemma~\ref{l:ciagi}.\bez

\section{Proofs of  Theorem~\ref{p:uniflog2A} and Corollary~\ref{p:uniflog2}}

Theorem~\ref{p:uniflog2A} is an immediate consequence of the following lemma:
\begin{Lemma} \label{l:momocount}  Let $C=\{g_k\colon \,k\geq1\}\subset G$, where the  nil-rotations $L_{g_k}(g\Gamma)=g_kg\Gamma$ are ergodic. Then, the homeomorphism $T\colon  C\times G/\Gamma\to C\times G/\Gamma$, $T(g_k,g'\Gamma)=(g_k,g_kg'\Gamma)$ satisfies the strong LOMO property.
\end{Lemma}
\begin{proof} Because of our assumptions on the set $C$, the homeomorphism $T$ has only countably many ergodic measures. Indeed, it follows from~\cite{Gree,Le} that a nil-rotation $L_{g_k}$ is ergodic if and only if it is uniquely ergodic. Hence, for each $g_k\in C$, there is exactly one measure invariant on the fiber over $g_k$. Hence, by the work of Frantzikinakis and Host~\cite[Theorem 1.1]{Fr-Ho}, it satisfies the logarithmic Sarnak conjecture. In fact, as noticed in~\cite[Corollary 1.2]{Go-Le-Ru}, the theorem of Frantzikinakis and Host implies that all zero entropy systems with a countable set of ergodic measures satisfy the logarithmic strong LOMO property. It follows that $T$ satisfies the logarithmic strong LOMO property.\end{proof}

\begin{proof}[Proof of Corollary~\ref{p:uniflog2}]
Let $C=\{\alpha_k\colon  k\in\N\}\subset \T$ be closed with all $\alpha_k$ irrational.
Consider the following groups of $(d+1)\times(d+1)$ upper triangular matrices:
\[
G=\begin{pmatrix}
1	&\Z	&\Z	&\cdots 	&\Z&\Z	&\R\\
0	&1	&\Z	&\cdots	&\Z&\Z	&\R\\
0	&0	&1	&\cdots	&\Z&\Z	&\R\\
\vdots	&\vdots	&\vdots &\ddots	&\ddots	&\ddots	&\vdots\\
0	&0	&0	&\cdots		&1		&\Z		&\R\\
0	&0	&0	&\cdots		&0		&1		&\R\\
0	&0	&0	&\cdots		&0		&0		&1\\
\end{pmatrix}
,~~~~~~~
\Gamma=
\begin{pmatrix}
1	&\Z	&\Z	&\cdots 	&\Z&\Z	&\Z\\
0	&1	&\Z	&\cdots	&\Z&\Z	&\Z\\
0	&0	&1	&\cdots	&\Z&\Z	&\Z\\
\vdots	&\vdots	&\vdots &\ddots	&\ddots	&\ddots	&\vdots\\
0	&0	&0	&\cdots		&1		&\Z		&\Z\\
0	&0	&0	&\cdots		&0		&1		&\Z\\
0	&0	&0	&\cdots		&0		&0		&1\\
\end{pmatrix}.
\]

Note that $G$ is a $d$-step nilpotent Lie group generated by the connected component of the identity and a finitely generated torsion-free subgroup, and $\Gamma$ is a discrete and cocompact subgroup of $G$.
Through the diffeomorphic map
\[
\varphi\colon(x_1,\ldots,x_{d-1},x_d)\mapsto
\begin{pmatrix}
1&0&\cdots&0&x_d\\
0&1&\cdots&0&x_{d-1}\\
\vdots&\vdots&\ddots&\vdots&\vdots\\
0&0&\cdots&1&x_1\\
0&0&0&0&1
\end{pmatrix}\Gamma
\]
we can identify the nilmanifold $G/\Gamma$ with the torus $\T^d$.
Also, define
\[
g_k
=
\begin{pmatrix}
1&1&0&\cdots&0&0\\
0&1&1&\cdots&0&0\\
\vdots&\vdots&\ddots&\ddots&\vdots&\vdots\\
0&0&\cdots&1&1&0\\
0&0&\cdots&0&1&\alpha_1\\
0&0&0&0&0&1
\end{pmatrix}
\]
and note that the nil-rotation induced by $g_k$ on $G/\Gamma$ is ergodic and congruent, via $\varphi$, to the affine linear transformation
$T_k(x_1,x_2,\ldots, x_d)=(x_1+\alpha_k,x_2+x_1,\ldots, x_d+x_{d-1})$ on $\T^d$.

By applying Theorem~\ref{p:uniflog2A} to the nilmanifold $G/\Gamma$ and the countable closed set of ergodic nil-rotations $\{g_k\colon k\in\N\}\subset G$, and invoking the isomorphism $\varphi$, we get for any continuous function $f\colon \T^d\to\C$ that
\[
\lim_{H\to\infty}\limsup_{M\to\infty}~\E_{m\leq M}^{\log}~ \sup_{k\in\N}~\sup_{x\in\T^d}~\biggl|\E_{h\leq H}\lio(m+h)f(T_k^{m+h}x) \biggr|=0.
\]
Iterating the transformation $T_k$ yields
\begin{equation}
\label{eqn_cor_poly_cntbl_lead_coef_1}
T_k^n(x_1,x_2,\ldots, x_d)=
\biggl(
n\alpha_k+x_1,~\ldots,~~\binom{n}{d} \alpha_k +\sum_{i=1}^d \binom{n}{d-i}x_i
\biggr).
\end{equation}
So, by selecting $x=(x_1,\ldots,x_d)\in\T^d$ appropriately, we can achieve in the last coordinate of $T_k^n(x_1,x_2,\ldots, x_d)$ any polynomial of degree $d$ whose leading coefficient is $\alpha_k$. The conclusion of Corollary~\ref{p:uniflog2} now follows from~\eqref{eqn_cor_poly_cntbl_lead_coef_1} applied to the function $f(x_1,\ldots,x_d)=e(x_d)$.
\end{proof}

\subsection{\texorpdfstring{What happens if $C$ contains a rational number?}{What happens if C contains a rational number?}} \label{ss:zeronalezy}

We will now show that if~\eqref{unif7c} holds for some $t\geq 2$ and some set $C$ containing a rational number, then~\eqref{unif7c} holds with $t-1$ in place of $t$ with the full set $C=\T$.
So, while
\beq
\label{eqn_packdim_2}\lim_{H\to\infty}\limsup_{M\to\infty}\EE_{m\leq M}^{\log}\sup_{\alpha\in C}\big|\EE_{h\leq H}\lio(m+h)e(\alpha h^t)\big|=0\eeq
holds for $t=1$ and all closed sets $C\subset \T$ with $\Leb(C)=0$ by Theorem~\ref{thm:packingdimension}, we do not expect that our methods can prove~\eqref{eqn_packdim_2} for all closed sets $C\subset \T$ with $\Leb(C)=0$ in the case $t\geq 2$.

Let $a\in \mathbb{Z}$, $q\in \mathbb{N}$ be such that~\eqref{unif7c} holds with $t$ for the set $C=\left\{\frac{a}{q}\right\}$. Then, since the function $n\mapsto e\left(-\frac{a}{q}n^{t}\right)$ is $q$-periodic, we have a Fourier expansion
$$
1=\sum_{b=1}^{q}c_be\left(\frac{a}{q}n^{t}+\frac{bn}{q}\right)
$$
for some complex numbers $c_b$. Multiplying both sides by $e(Q(n))$, where $Q$ is any polynomial of degree $\leq t-1$, we see from the triangle inequality that~\eqref{unif7c} holds with $t-1$ in place of $t$ for the full set $C=\T$.

\section{Proof of Theorem~\ref{p:unif1}}
\subsection{Some Cantor sets and rigidity}

We are interested in $C\subset\T$ which are closed and for which there exists a sequence $(q_n)$ such that, for any $\alpha\in C$, we have
\beq\label{ri1}\lim_{n\to \infty}\|q_n\alpha\|=0.\eeq

\begin{Remark} In general, consider any strictly increasing sequence $k_n$ and let
$$C=\bigcap_{n\geq1}\{\alpha\in\T\colon \,\|2^{k_n}\alpha\|\leq 1/\ell_n\}.$$
Then the set $C$ is closed and it satisfies~\eqref{ri1} if $\ell_n\to\infty$. Some information about Hausdorff dimension of such sets can be found in~\cite{Li-Ra}.
In Appendix~\ref{App:B}, using  rather standard tools, we will present constructions of Cantor sets satisfying~\eqref{ri1} and having full Hausdorff dimension.
\end{Remark}

\begin{Lemma} \label{l:ri2} If $C$ satisfies~\eqref{ri1} then for  all invariant measures $\nu$ of the homeomorphism $T(x,y)=(x,x+y)$ acting on $C\times\T$  the sequence $(q_n)$ is a rigidity time for $(C\times\T,\nu,T)$.\end{Lemma}
\begin{proof} For each $\alpha\in C$, on $\{\alpha\}\times\T$, the homeomorphism $T$ acts as the rotation by $\alpha$ and the observation follows by~\eqref{ri1}   ($T^{q_n}(\alpha,y)=(\alpha,y+q_n\alpha)\to (\alpha,y)$ pointwise).\end{proof}

\subsection{Rigidity and a proof of Theorem~\ref{p:unif1}}
\begin{Lemma}\label{l:ri3} Let us fix $(q_n)$ with bounded prime volume.
If $(X,T)$ is a topological system such that {all} invariant measures yield rigidity, with $(q_n)$ being a rigidity time, then $(X,T)$ satisfies the strong LOMO property.\footnote{Since we are talking about rigidity along a fixed sequence, the assumption ``all'' can be replaced with ``all ergodic''.}\end{Lemma}
\begin{proof} We need to prove that in the orbital models (extended by the three-point space $\mathbb{\A}=\{e(j/3)\colon \,j=0,1,2\}$, cf.\ the proof of~\cite[Corollary~9]{Ab-Ku-Le-Ru}) obtained by $(b_k)$ and $(x_k)$, the points are quasi-generic only for $(q_n)$-rigid measures and then we use~\cite[Theorem~2.1]{Ka-Le-Ra}.

So let $Y=(X\times\mathbb{A})^{\N}$ and let $S$ be the left shift. Let
$$
\underline{y}=(y_n),\;y_n=T^{n-b_k}x_k\text{ for }b_k\leq n<b_{k+1}$$
and $\underline{a}=(a_n)$ with $a_n=a_k$ for $b_k\leq n<b_{k+1}$. Clearly the set $Z\coloneqq \{\underline{v}\in Y\colon \,(v_1,a_1)=(Tv_0,a_0)\}$ is closed. Hence, because of the properties of $(\underline{y},\underline{a})$,
$$
\Big(\frac1{N_r}\sum_{n<N_r}\delta_{S^n(\underline{y},\underline{a})}\Big)(Z)\to1.$$
Basic properties of weak-$\ast$-topology then yield that if $\rho$ is the limit of these empiric measures then $\rho$ is $S$-invariant and
$$
\rho(\{((x,a),(Tx,a),(T^2x,a),\ldots)\colon \, x\in X,a\in\mathbb{A}\})=1.$$
Let us see now what is the projection $\rho_1$ of $\rho$ on the first coordinate $X\times \mathbb{A}$: namely, it is the limit of   (assuming that $b_K<N_r<b_{K+1}$)
$$
\frac1{N_r}\Big(\sum_{j<K}\sum_{b_j\leq n<b_{j+1}}\delta_{(T^nx_j,a_j)}+\sum_{b_K\leq n<N_r}\delta_{(T^nx_K,a_K)}\Big),$$
so we obtain a measure which is $T\times \textnormal{Id}$-invariant. It is hence a joining of a measure which is $T$-invariant and of a measure on $\mathbb{A}$. Since these two measures are $(q_n)$-rigid, $\rho$ is $(q_n)$-rigid. Now, by the above, $\rho$ is just the image of $\rho_1$ by the embedding
$$
(x,a)\mapsto ((x,a),S(x,a),S^2(x,a),\ldots),$$
so also $\rho$ is $(q_n)$-rigid.  For the remaining points in the closure of the orbit of $(\underline{y},\underline{a})$, we apply the same argument as in~\cite{Ab-Le-Ru} or~\cite{Ab-Ku-Le-Ru}.\footnote{If $n_j\to\infty$ and $\underline{v}=\lim_{j\to\infty} S^{n_j}(\underline{y},\underline{a})$, then
 for some $x_1,x_2\in X$ and $a_1,a_2\in\mathbb{A}$, we have
$\underline{v}=((x_1,a_1),(Tx_1,a_1),\ldots,(T^\ell x_1,a_1),(x_2,a_2),(Tx_2,a_2),\ldots)$ for some $\ell\geq0$.}
\end{proof}

\begin{proof}[Proof of Theorem~\ref{p:unif1}]
The result follows from Lemmas~\ref{l:ri2} and~\ref{l:ri3}.
\end{proof}

\appendix

\section{\texorpdfstring{Proof of Lemma~\ref{l:ciagi}}{Proof of Lemma 1.7}}\label{App:A}

We first show that~\eqref{eq:normed1} implies~\eqref{eq:normed2}. Let $(b_k)$ be an increasing sequence with $d(\{b_k\colon \,k\geq 1\})=0$. Then we have $\lim_{K\to \infty} K/b_K=0$. Let $H\geq 1$ be an integer. We have
\begin{align*}
\left\|\sum_{b_k\leq n<b_{k+1}}z_n-\frac{1}{H}\sum_{b_k\leq m<b_{k+1}}\sum_{h\leq H}z_{m+h}\right\|\ll H.
\end{align*}
Hence,
\begin{align*}
\limsup_{K\to \infty}\frac{1}{b_K}\sum_{k<K}\left\|\sum_{b_k\leq n<b_{k+1}}z_n\right\|&\leq \limsup_{K\to \infty}\frac{1}{Hb_K}\sum_{k<K}\sum_{b_k\leq m< b_{k+1}}\left\|\sum_{h\leq H}z_{m+h}\right\|\\
&=\limsup_{K\to \infty}\frac{1}{Hb_K}\sum_{m<b_{K}}\left\|\sum_{h\leq H}z_{m+h}\right\|.
\end{align*}
Letting $H\to \infty$ shows that~\eqref{eq:normed1} implies~\eqref{eq:normed2}.

We now show that~\eqref{eq:normed2} implies~\eqref{eq:normed1}. Observe that if~\eqref{eq:normed1} fails, then there is some increasing function $H\colon  \mathbb{N}\to \mathbb{N}$ with $H(m)\leq \log m+1$ (say) and some increasing sequence $(M_i)$ satisfying $M_{i+1}> M_i^2$ such that
\begin{align}\label{eq:bk3}
\limsup_{i\to\infty}\mathbb{E}_{m\leq M_i}\frac{1}{H(M_i)}\left\|\sum_{m\leq k\leq m+H(M_i)}z_{k}\right\|>0.
\end{align}

By the pigeonhole principle, for each $i\geq 1$ there exists $a(i)\in [1,2H(M_i)]\cap \mathbb{N}$ such that the left-hand side of~\eqref{eq:bk3} is
\begin{align}\label{eq:bk4}
\ll  \limsup_{i\to\infty}H(M_i)\mathbb{E}_{m\leq M_i}\raz_{m\equiv a(i)\pmod{2H(M_i)}}\frac{1}{H(M_i)}\left\|\sum_{m\leq k\leq m+H(M_i)}z_{k}\right\|.
\end{align}
By passing to a subsequence if necessary, we may assume that $\lim_{i\to \infty}\frac{a(i)}{H(M_i)}\coloneqq a_0\in [0,2]$ exists. Now we see that~\eqref{eq:bk4} is
\begin{align}\label{eq:bk5}\begin{split}
&\limsup_{i\to\infty}\frac{1}{M_i}\sum_{\ell\leq \frac{M_i}{2H(M_i)}}\left\|\sum_{(2\ell+a_0)H(M_i)\leq k\leq (2\ell+1+a_0)H(M_i)}z_{k}\right\|\\
=&\limsup_{i\to\infty}\frac{1}{M_i}\sum_{(\frac{M_i}{2H(M_i)})^{9/10}\leq \ell\leq \frac{M_i}{2H(M_i)}}\left\|\sum_{(2\ell+a_0)H(M_i)\leq k\leq (2\ell+1+a_0)H(M_i)}z_{k}\right\|.
\end{split}
\end{align}

Now, let $M_{\ell}^{*}$ denote the least element of the sequence $(M_i)$ that is $\geq \ell$. Then $M_{\ell}^{*}=M_i$ for all $\ell\in [\big(\frac{M_i}{2H(M_i))}\big)^{9/10}, M_i]$ (recalling that $M_{i}>M_{i-1}^2$). Define a strictly increasing sequence $(b_k)$ by $b_{2k}=\lfloor (2k+a_0)H(M_k^{*})\rfloor$, $b_{2k+1}=\lfloor (2k+1+a_0)H(M_k^{*})\rfloor$. Then $b_{k+1}-b_k\to \infty$ as $k\to \infty$, so $d(\{b_k\colon \,\, k\geq 1\})=0.$ Also, we have $b_{\lfloor M_i/(2H(M_i))\rfloor}\asymp M_i$. Hence,~\eqref{eq:normed2} with this sequence $(b_k)$ contradicts~\eqref{eq:bk3}.

The case of logarithmic averages is proved completely analogously.

\section{Assumptions on nilpotent Lie groups}\label{a:nowy}
The aim of this section is to prove that a wide range of nilpotent Lie groups can be realized as a factor of a subgroup of a connected, simply connected nilpotent Lie group. The precise statement is as follows.

\begin{Prop}\label{p:embeddingnilLie}
Let $G$ be a nilpotent Lie group, $\Gamma$ a discrete cocompact subgroup of $G$, and assume that $G$ is spanned by the connected component of the identity element and finitely many other group elements. Then there exists a connected and simply connected Lie group $\widehat{G}$ with the same nilpotency step as $G$, a closed Lie subgroup $\widetilde{G}$ of $\widehat{G}$ and a surjective Lie group homomorphism $\widetilde{\pi}\colon \widetilde{G}\to G$ such that $\widehat{\Gamma}=\widetilde{\pi}^{-1}(\Gamma)$ is a cocompact lattice in $\widehat{G}$. In particular, the nilmanifold $G/\Gamma$ is isomorphic to the nilmanifold $\widetilde{G}/\widehat{\Gamma}$ which embeds as a subnilmanifold into the nilmanifold $\widehat{G}/\widehat{\Gamma}$.
\end{Prop}

In what follows let $G^\circ$ denote the connected component of the identity element of a nilpotent Lie group $G$.
If $G^\circ$ is simply connected and $G/G^\circ$ is a finitely generated and torsion-free group then the conclusion of Proposition~\ref{p:embeddingnilLie} follows directly from \cite[Theorem 2.20]{Ra}. 
In the case when $G/G^\circ$ is a finitely generated abelian group, Proposition~\ref{p:embeddingnilLie} was proved in \citep[Lemma 7, p.~156]{Ho-Kr}. 
The main ingredient in our proof of Proposition~\ref{p:embeddingnilLie} is a generalization of \citep[Lemma 7, p.~156]{Ho-Kr} from the abelian case to the nilpotent case given in the next lemma. The notion of a free nilpotent group is defined in Section~\ref{sec_B1}.

\begin{Lemma}\label{l:HostKra} Let $G$ be an $s$-step nilpotent Lie group and assume that $G$ is spanned by $G^\circ$ and $q$ elements $\tau_1,\ldots,\tau_q$. Then there exist a simply connected $s$-step nilpotent Lie group $\widetilde{G}$ and a surjective Lie group homomorphism $\widetilde\pi:\widetilde{G}\to G$ whose kernel ${\rm ker}(\widetilde\pi)$ is discrete. Moreover, there exist $\widetilde\tau_1,\ldots,\widetilde\tau_q\in \widetilde{G}$ such that $\widetilde{\pi}(\widetilde{\tau}_i)=\tau_i$ for $i=1,\ldots,q$, $\widetilde{G}$ is spanned by $\widetilde{G}^\circ$ and $\widetilde\tau_1,\ldots,\widetilde\tau_q$, and the group $\langle \widetilde\tau_1,\ldots,\widetilde\tau_q \rangle$ is a free $s$-step nilpotent group and isomorphic to $\widetilde{G}/\widetilde{G}^\circ$.
In particular, $\widetilde{G}/\widetilde{G}^\circ$ is a finitely generated and torsion-free group.
\end{Lemma}

\begin{proof}[Proof of Proposition~\ref{p:embeddingnilLie} assuming Lemma~\ref{l:HostKra}]
Let $G$  and $\Gamma$ be as in the statement of Proposition~\ref{p:embeddingnilLie}.
In view of Lemma~\ref{l:HostKra}, there exists a simply connected Lie group $\widetilde{G}$ of the same nilpotency step as $G$ such that $\widetilde{G}/\widetilde{G}^\circ$ is a finitely generated torsion-free group, and a surjective Lie group homomorphism $\widetilde\pi:\widetilde{G}\to G$ whose kernel ${\rm ker}(\widetilde\pi)$ is discrete.
Define $\widetilde{\Gamma}=\widetilde\pi^{-1}(\Gamma)$ and note that $\widetilde{\Gamma}$ is a discrete and cocompact subgroup of $\widetilde{G}$ and the nilmanifolds $\widetilde{G}/\widetilde{\Gamma}$ and $G/\Gamma$ are isomorphic.
We can now apply \cite[Theorem 2.20]{Ra} and embed $\widetilde{G}$ into a connected, simply connected nilpotent Lie group $\hat{G}$ of the same nilpotency step and such that the induced embedding $\widehat{\Gamma}$ of $\widetilde{\Gamma}$ into $\hat{G}$ remains a discrete and cocompact subgroup of $\hat{G}$. 
\end{proof}

\subsection{Free nilpotent cover}
\label{sec_B1}
Given a group $H$ let $H_n$ denote the $n$th term of the {\em lower central series} of $H$, that is $H_1=H$ and $H_{n+1}=[H_n,H]$,\footnote{Given two subsets $L,M$ of $H$ we denote by $[L,M]$ the subgroup of $H$ generated by all commutators $[l,m]=lml^{-1}m^{-1}$ with $l\in L, m\in M$. $[L,M]$ is  a normal subgroup of $H$ whenever $L,M$ are normal.} $n\in\N$. By definition, The group $H$ is {\em nilpotent (of step $\le n$)} if $H_{n+1}=\{e\}$ for some $n$. It is easy to check that $H_n$ is the subgroup of $H$ generated by all commutators of the form
$$
[\ldots[[g_1,g_2],g_3],\ldots,g_n],
$$
where $g_1,\ldots,g_n\in H$. Note that, for every group $H$, the factor $H/H_{n+1}$ is a nilpotent group (of step $\le n$).

For every $n\in\N$ there exists a surjective homomorphism $H/H_{n+1}\rightarrow H/H_n$ with the kernel isomorphic to $H_n/H_{n+1}$. In other words, there is a short exact sequence
\begin{equation}\label{eq:exseq}
\{e\}\rightarrow H_n/H_{n+1}\rightarrow H/H_{n+1}\rightarrow H/H_n\rightarrow\{e\}.
\end{equation}

If $F$ is a free group in $q$ generators then $F/F_{n+1}$ is called a \emph{free $n$-step nilpotent group in $q$ generators}.

\begin{Lemma}\label{c:SK1}
For every finitely generated $n$-step nilpotent group $H$ there exists a free finitely generated $n$-step nilpotent group $\widetilde{H}$ and a surjective group homomorphism $\widetilde{H}\rightarrow H$.
\end{Lemma}

\begin{proof}
Assume that  $g_1\ldots g_r$ generate $H$ and $H$ is nilpotent of step $n$. Let $F$ be the free group with $r$ free generators $f_1,\ldots, f_r$. Then the mapping $f_i\mapsto g_i$, $i=1,\ldots, r$ induces a surjective group homomorphism from $F/F_{n+1}$ to $H$. The group $\widetilde{H}=F/F_{n+1}$ is a free finitely generated nilpotent group of step $n$, finishing the proof.
\end{proof}

\begin{Remark}
The groups $F/F_{n+1}$ are free objects in the variety of the nilpotent groups of degree $\le n$, see \cite[Chap. VI]{KarMier} and \cite{Tao}. 
Recall also the well-known fact that $F/F_2$ - the abelianization of a free group $F$ - is free abelian.
\end{Remark}

\begin{Lemma}\label{lem:free_nilp_tor_free}
Let $F$ be a free group in $q$ generators. Then $F/F_n$ is torsion-free for every $n\in\N$.
\end{Lemma}

\begin{proof}
 It follows from \cite[Theorem 5.12]{MagKarSol} that  $F_{n}/F_{n+1}$ is a free abelian fintely generated group for every $n\in\N$, therefore also torsion-free (see also \cite{Oml}).
If both the end terms of a short exact sequence of groups are torsion-free, then the middle term is also torsion-free\footnote{Assume that $B$ is a normal subgroup of $A$ and both $B$ and $A/B$ are torsion-free. Let $a\in A$ be an element of finite rank, say, $a^r=e$. Then the coset of $a$ is of finite rank in $A/B$, so, since $A/B$ is torsion-free, $a\in B$.  But $B$ is torsion-free, so $a=e$. }. We apply this observation to the sequences (\ref{eq:exseq}) with $H=F$:
\begin{equation}
\begin{array}{c}
\{e\}\rightarrow F/F_{2}\rightarrow F/F_{2}\rightarrow F/F_1=\{e\}\rightarrow\{e\},\\
\{e\}\rightarrow F_2/F_{3}\rightarrow F/F_{3}\rightarrow F/F_2\rightarrow\{e\},\\
\{e\}\rightarrow F_3/F_{4}\rightarrow F/F_{4}\rightarrow F/F_3\rightarrow\{e\},\\
\ldots
\end{array}
\end{equation}
and we derive the lemma  by induction on $n$.
\end{proof}

\subsection{Proof of Lemma~\ref{l:HostKra}}

\begin{proof}[Proof of Lemma~\ref{l:HostKra}]
Recall that that $G$ is an $s$-step nilpotent Lie group generated by its connected component\footnote{Note that a connected Lie group is automatically path connected.} $G^\circ$ and $\langle \tau_1,\dots,\tau_q\rangle$  (both subgroups being obviously nilpotent of step $\leq s$). 

Denote by $\widetilde{G}^\circ$ the universal cover of $G^\circ$ with the homomorphism $\widetilde{\pi}_0:\widetilde{G}^\circ\to G^\circ$. Let $\phi_j\in{\rm Aut}(G^\circ)$ be given by $\phi_j(g)=\tau_jg\tau_j^{-1}$, $j=1,\ldots, q$. By the universal property of the universal cover, each such $\phi_j$ lifts uniquely to an automorphism $\widetilde\phi_j$ of $\widetilde{G}^\circ$.
Let $H$ be the group generated by $\widetilde\phi_j$, $j=1,\ldots, q$ and, using Lemma~\ref{c:SK1}, let $\widetilde{H}$ denote the free nilpotent cover of $H$. Let $\rho\colon \widetilde{H}\to H$ be the induced factor map and let $\varphi_1,\ldots,\varphi_q$ denote the generators of $\widetilde{H}$ satisfying $\rho(\varphi_j)=\widetilde\phi_j$.

Note that $\rho(\varphi)$ is an automorphism of $\widetilde{G}^\circ$ for every $\varphi\in\widetilde{H}$. So we can define the semi-direct product $\widetilde{G}:=\widetilde{G}^\circ\rtimes \widetilde{H}$, where
$$
(g,\varphi)\cdot(g',\varphi')=(g\cdot \rho(\varphi)(g'),\varphi\cdot \varphi'),\qquad\forall (g,\varphi),(g',\varphi')\in \widetilde{G}^\circ\times \widetilde{H}.
$$
Observe that:\\
(a) $\widetilde{G}^\circ\times\{e_{\widetilde{H}}\}$ is a normal subgroup of $\widetilde G$;\\
(b) as a topological space  $\widetilde{G}=\widetilde{G}^\circ\times \widetilde{H}$ (in particular, $\widetilde{G}^\circ\times\{e_{\widetilde{H}}\}$ is an open subgroup);\\
(c) If $\widetilde\tau_j=(e_{\widetilde{G}^\circ},\varphi_j)$ then $\langle\widetilde\tau_1,\ldots,\widetilde\tau_q\rangle $ is isomorphic to $\widetilde{H}$.

It follows from (a) and (b) that $\widetilde{G}^\circ\times\{e_{\widetilde{H}}\}$ is the connected component of $e_{\widetilde{G}}$ and it follows from (c) that $\widetilde{G}$ is spanned by $\widetilde{G}^\circ\times\{e_{\widetilde{H}}\}$ and $\widetilde\tau_1,\ldots, \widetilde\tau_q$.
Therefore, $\widetilde{G}/\widetilde{G}^\circ$ is isomorphic to $\widetilde{H}=\langle \widetilde\tau_1,\ldots,\widetilde\tau_q \rangle$, which is a free $s$-step nilpotent group. 
In particular, this group is torsion-free due to Lemma~\ref{lem:free_nilp_tor_free}.

Every element of $\varphi$ of $\widetilde{H}$ can be written as $\varphi=\prod_{k=1}^K\varphi_{j_i}$ with $j_i\in\{1,\ldots,q\}$, $i=1,\ldots,K$. Then $\widetilde\pi:\widetilde{G}\to G$,
\begin{equation}
\label{eqn_phitilde_def}
\widetilde\pi(g,\varphi)=\widetilde\pi_0(g)\prod_{i=1}^K\tau_{j_i}
\end{equation}
is a well defined homomorphism satisfying $\widetilde{\pi}(\widetilde{\tau}_i)=\tau_i$ for $i=1,\ldots,q$. 
We have the following:\\
 {\bf Claim I.}  ${\rm ker}(\widetilde\pi)$ is discrete.

 Indeed, let $(g,\varphi)\in {\rm ker}(\widetilde\pi)$, $\varphi=\prod_{k=1}^K\varphi_{j_i}$. Then, by \eqref{eqn_phitilde_def}, $\widetilde\pi_0(g)$ belongs to a countable subgroup generated by $\tau_{j_i}$. But $\widetilde\pi_0$ is countable to~1, hence, $g$ belongs to a countable subset of $\widetilde{G}^\circ$. Since ${\rm ker}(\widetilde\pi_0)$ is also closed, it must be discrete.\\
 {\bf Claim II.} $\widetilde{G}$ is $s$-step nilpotent.

 Indeed, since $\widetilde\pi(\widetilde{G}_{s+1})\subset G_{s+1}=\{e_G\}$ as $\widetilde{\pi}$ is a homomorphism, $\widetilde{G}_{s+1}\subset {\rm ker}(\widetilde\pi)$ must also be discrete. On the other hand this commutator is connected (see below), so $\widetilde{G}_{s+1}$ is trivial and therefore
 $\widetilde{G}$
  is an
  $s$-step nilpotent group.

To complete the proof of the proposition we need to show that $\widetilde{G}_{s+1}$ is connected. In our situation, $\widetilde{G}=\widetilde{G}^\circ\rtimes\widetilde{H}$, where $\widetilde{G}^\circ$ is (normal) path connected, and $\widetilde{H}$ is at most $s$-step nilpotent.

We take $t\geq s+1$. Then for each fixed $(q_1,\ldots q_t)\in \{e_{\widetilde{G}^\circ}\}\times\widetilde{H}^{t}$ we consider the map $\beta_{q_1,\ldots,q_t}:(\widetilde G^\circ\times \{e_{\widetilde{H}}\})^t\to \widetilde G$ given
$$
\beta_{q_1,\ldots,q_t}(a_1,\ldots, a_t)=[a_1q_1[a_2q_2[...[a_{t-1}q_{t-1},a_tq_t]]]]].$$
Then  $\beta_{q_1,\ldots,q_t}$ is continuous for each choice of $(q_1,\ldots, q_t)$. We use now Lemma~2 (p. 12) of \cite{Ho-Kr} to obtain that $\widetilde G_{s+1}$ is spanned by the union over all $t$-tuples $(q_1,\ldots,q_t)$, $t\geq s+1$, of the sets $\beta_{q_1,\ldots, q_t}((\widetilde G^\circ\times \{e_{\widetilde{H}}\})^t)$.
It follows that $\widetilde{G}_{s+1}$ is the group generated by a union of sets each of which is pathwise connected. However, each of these sets contains $e_{\widetilde{G}}$, by taking $a_i=e_{\widetilde{G}}$ and using the fact that $\widetilde{H}$ is $s$-step nilpotent  (so this commutator equals $\{e_{\widetilde{G}}\}$). By the first observation in the proof of Lemma~5 (p.\ 155) \cite{Ho-Kr} we conclude that $\widetilde{G}_{s+1}$ is pathwise connected.
\end{proof}

\section{Construction of full Hausdorff dimension Cantor sets with a certain Diophantine approximation property}\label{App:B}

In this appendix, we prove the following complement to Theorem~\ref{p:unif1}.

\begin{Prop}\label{prop_hausdorff}
There exists a closed set $C\subset [0,1]$ of full Hausdorff dimension such that, for some sequence $(q_n)$ of natural numbers, we have
$$\lim_{n\to \infty}\|q_n\alpha\|=0\text{ for  each }\alpha\in C$$
and
$$\sup_n\sum_{\substack{p\in \PP\\ p\mid q_n}}\frac1p<+\infty.$$
\end{Prop}

The proof is based on the following lemma. We follow~\cite{Fr-Le} based on~\cite{Fa} (see~\cite[Example~4.6]{Fa} and its proof), see also~\cite[Lemma~9]{Bu-Li-Ra}.

\begin{Lemma}\label{le_dimension}
 Let $(C_n)_{n\geq 1}\subset\mathbb{R}$ be a decreasing sequence of closed sets each of which is a finite union of pairwise disjoint closed intervals, called $n$-th level basic intervals. We assume that
each $n-1$-st level basic interval of $C_{n-1}$ includes at least $m_n\geq 2$ $n$-th level basic intervals of $C_n$. Also assume that the maximal length of $n$–th level basic intervals tends to zero when $n\to\infty$. Furthermore, assume that the gap between two consecutive $n$-th level basic intervals is at least $\vep_n$ (with $\vep_n>\vep_{n+1}>0$).

Then, the Hausdorff dimension ${\rm dim}_{\textnormal{H}}(C)$ of the intersection $C\coloneqq \bigcap_{n\geq1}C_n$ is at least $\liminf_{n\to\infty}\frac{\log(m_1\cdots m_{n-1})}{-\log(m_n\vep_n)}$.
\end{Lemma}

\begin{proof}
This is proved in~\cite[Section~6]{Fr-Le}.
\end{proof}

\begin{proof}[Proof of Proposition~\ref{prop_hausdorff}]
Fix a monotone sequence $0<\delta_n\to0$.
To construct sets $C$ of the desired form we will take a  sparse sequence  $(k_n)$ (how sparse this sequence is depends on $(\delta_n)$) and have at stage $n-1$ a closed set $C_{n-1}$ consisting of the union of small neighbourhoods $[\frac{j-\delta_{n-1}}{2^{k_{n-1}}},\frac{j+\delta_{n-1}}{2^{k_{n-1}}}]$ of $\frac{j}{2^{k_{n-1}}}$ for {\em some} values of $j$. So the distance between the $n-1$-st level basic intervals  is at least $\frac1{2^{k_{n-1}}}(1-2\delta_{n-1})$.
Now, if $k_n$ is large enough, we form the family of $n$-th level basic intervals and hence $C_n$ by first partitioning each $n-1$-st level basic interval $I$ into many intervals of the form $[\frac{r}{2^{k_n}},\frac{r+1}{2^{k_n}}]$ (two of these intervals may overlap $I$ only partially) and then around each point $\frac{r}{2^{k_n}}$ choosing a small interval $[\frac{r-\delta_{n}}{2^{k_{n}}},\frac{r+\delta_{n}}{2^{k_{n}}}]\subset I$.

Note that for $x\in C_n$, we have $\|2^{k_n}x\|\leq \delta_n$.
Now, each $n-1$-st level basic interval contains at least
$$
\frac{\delta_{n-1}\frac1{2^{k_{n-1}}}}{\frac1{2^{k_n}}}=\delta_{n-1} 2^{k_n-k_{n-1}}\eqqcolon m_n$$
$n$-th level basic intervals. Moreover, the distance between any consecutive $n$-th level basic intervals is
$$
\geq \vep_n\coloneqq \frac1{2^{k_n}}(1-2\delta_n).$$
It follows that
$$
-\log(\vep_nm_n)=k_{n-1}-\log(\delta_{n-1}(1-2\delta_n))$$
and
$$
\log(m_1\cdots m_{n-1})=(k_{n-1}-1)+\sum_{j=1}^{n-1}\log\delta_{j-1}.$$

Now, Lemma~\ref{le_dimension} gives
$$
{\rm dim}_{\textnormal{H}}(C)\geq \liminf_{n\to\infty}\frac{(k_{n-1}-1)+\sum_{j=1}^{n-1}\log\delta_{j-1}}
{k_{n-1}-\log(\delta_{n-1}(1-2\delta_n))}=1$$
if $k_n$ is growing fast enough compared to $1/\delta_n$. The claim follows, since the sequence $(2^{k_n})$ certainly has bounded prime volume.
\end{proof}

\vspace{2ex}

\noindent
{\bf Acknowledgments}:
The results of the paper have been obtained in a series of discussions of the authors in the framework of the Special Year on Dynamics, Additive Number Theory and Algebraic Geometry at the Institute for Advanced Study (Fall 2022). The research of the second author was partially supported by  Narodowe Centrum Nauki grant UMO-2019/33/B/ST1/00364. The research of the fourth author was supported by a von Neumann Fellowship  and funding from European Union's Horizon
Europe research and innovation programme under Marie Sk\l{}odowska-Curie grant agreement No 101058904. The third and fourth author were supported by NSF grant \texttt{DMS-1926686}. We thank the anonymous referees for helpful comments. 

{Our special thanks go to Stanis\l aw Kasjan for showing us the proof of Lemma~\ref{lem:free_nilp_tor_free} and to one of the referees for providing the proof of Theorem~\ref{th:assum}}.

\end{document}